\documentclass[12pt]{amsart}

\usepackage[dvips]{color}
\usepackage{amsmath}
\usepackage{amsxtra}
\usepackage{amscd}
\usepackage{amsthm}
\usepackage{amsfonts}
\usepackage{amssymb}
\usepackage[mathscr]{eucal}
\usepackage{epsfig}
\usepackage{graphics}
\usepackage{accents}
\usepackage{physics}
\usepackage{hyperref}
\usepackage{mathtools}
\usepackage{tikz}
\usepackage[all]{xy}
\usepackage{here} 
\usepackage{stmaryrd}
\usepackage{tikz-cd}
\numberwithin{equation}{section}
\newtheorem{thm}{Theorem}[section]
\newtheorem{prop}[thm]{Proposition}
\newtheorem{lem}[thm]{Lemma}
\newtheorem{rem}[thm]{Remark}
\newtheorem{cor}[thm]{Corollary}

\newtheorem{dfn}[thm]{Definition}

\newcommand{\Z}{{\mathbb Z}}

\newcommand{\B}{{\mathcal B}}
\newcommand{\E}{{\mathcal E}}

\newcommand{\cX}{\mathcal{X}}
\renewcommand{\S}{\mathcal{S}}

\newcommand{\Smn}{\mathcal{S}_{m|n}}

\newcommand{\cZ}{\mathcal{Z}}

\newcommand{\bs}{\boldsymbol}

\newcommand{\gl}{\mathfrak{gl}}

\renewcommand{\sl}{\mathfrak{sl}}

\newcommand{\slh}{\widehat{\mathfrak{sl}}}

\newcommand{\s}{\mathbf{s}}

\newcommand{\Sym}{\mathrm{Sym}}

\newcommand{\glhs}{ \widehat{\gl}_{\s}}
\newcommand{\Es}{\E_{\s}}
\newcommand{\slhs}{ \widehat{\sl}_{\s}}

\newcommand{\cY}{\widehat{\mathcal{Y}}}

\newcommand{\U}{U}

\DeclareMathOperator\eva{ev}

\newcommand{\Ug}{U_q\widehat{\mathfrak{gl}}_{m|n}}
\newcommand{\Uge}{\widetilde{U}_q\widehat{\mathfrak{gl}}_{m|n}}

\newcommand{\lb}[1]{\llbracket #1 \rrbracket}
\newcommand{\hTa}{\widehat{\tau}}
\newcommand{\hT}{\widehat{T}}
\newcommand{\hX}{\widehat{\mathcal{X}}}

\begin{document}
\begin{title}[Braid actions]{Braid actions on quantum toroidal superalgebras}
\end{title}
\author{Luan Bezerra and Evgeny Mukhin}

\address{LB: Department of Mathematics,
	Indiana University -- Purdue University -- Indianapolis,
	402 N. Blackford St., LD 270,
	Indianapolis, IN 46202, USA} \email{luanpere@iupui.edu}

\address{EM: Department of Mathematics,
	Indiana University -- Purdue University -- Indianapolis,
	402 N. Blackford St., LD 270,
	Indianapolis, IN 46202, USA}\email{emukhin@iupui.edu}

\begin{abstract}
	We prove that the quantum toroidal algebras $\E_\s$ associated with different root systems $\s$ of $\gl_{m|n}$ type are isomorphic. We also show the existence of Miki automorphism of $\E_\s$, which exchanges the vertical and horizontal subalgebras.
	
	To obtain these results, we establish an action of the toroidal braid group on the
	direct sum $\oplus_\s\E_\s$ of all such algebras.
\end{abstract}
\maketitle
 {\vspace{-.5cm} \small Corresponding Author: Luan Bezerra}
 
\section{Introduction}
In this paper, we continue our study of the quantum toroidal algebras associated with the Lie superalgebra $\gl_{m|n}$ initiated in \cite{BM}. 

The root systems of $\sl_{m|n}$ are parameterized by sequences $\s=(s_1,\dots,s_{m+n})$, where $s_i=\pm1$, and $1$ occurs $m$ times, $-1$ occurs $n$ times.
We denote the algebra $\sl_{m|n}$ given in parity $\s$ by $\sl_{\s}$.
In \cite{BM}, we introduced the quantum toroidal algebra $\E_{m|n}$ corresponding to the standard parity $(1,\dots,1,-1,\dots,-1)$. In this paper, we define and study the quantum toroidal algebra $\E_\s$ associated with an arbitrary parity $\s$, see Definition \ref{defE}.

The idea for the definition is already described in \cite{BM}: we require $\Es$ to have vertical subalgebra $U_q^{ver}\slhs$, given in the current generators, and the horizontal subalgebra $U_q^{hor}\slhs$, given in the Chevalley generators, both in parity $\s$. In addition, we want our construction to be invariant under rotations $\hTa$ of the Dynkin diagram which connects $\Es$ with $\E_{\tau\s}$, where $\tau \s=(s_2,
\dots,s_{m+n},s_1)$. The algebra $\Es$ depends on parameters $q_1,q_2,q_3$, subject to $q_1q_2q_3=1$ with $q^2=q_2$.

It is natural to expect that all algebras $\Es$ should be isomorphic. In this paper, we prove that this is indeed so, see Corollary \ref{iso cor}. Similar statements are well known, see \cite{Y} for the case of quantum affine superalgebras, \cite{T} for the case of super  Yangians.

For the proof, we establish an action of the toroidal braid group $\widehat\B_{m+n}$ on $\E_\bullet=\oplus_\s\Es$. The group $\widehat\B_{m+n}$ is generated by $\hT_i, \cY_j,\hTa$, $i=1,\dots,m+n-1$, $j=0,\dots,m+n-1$. We already have $\hTa$. The automorphisms $\cY_j$ are given by explicit formulas, see \eqref{X hat}, \eqref{Y hat}. The main issue is the definition of $\hT_1$. We follow the logic of \cite{M}. We have an action of the extended affine braid group $\B_{m+n}$ on $U_q\slh_\bullet=\oplus_\s U_q\slhs$, see \cite{Y}. It turns out that the maps $T_{1,\s}: U_q^{ver}\slhs\to U_q^{ver}\slh_{\sigma_1\s}$ and $T_{1,\s}: U_q^{hor}\slhs\to U_q^{hor}\slh_{\sigma_1\s}$ agree on the common part $U_q\sl_\s=U_q^{ver}\slhs\cap U_q^{hor}\slhs$ and
give rise to the map of the whole algebra $\hT_{1,\s}:\Es\to \E_{\sigma_1\s}$, where $\sigma_1\s=(s_2,s_1,s_3,\dots,s_{m+n})$. The action of $\hT_1$ is not completely explicit, and we use various algebraic properties to check that it is well defined. In particular, using $\hTa$, we are able to reduce the checking to computations in the vertical subalgebra, for which we can use the results of \cite{Y}.

\medskip

As a byproduct, we also obtain the Miki automorphism, see Theorem \ref{mikithm}, which is central to the study of quantum toroidal algebras in the even case, see \cite{M}, \cite{FJMM}. The Miki automorphism is the  highly non-explicit automorphism which maps vertical and horizontal subalgebras to each other. Note that the isomorphism from $U_q\slhs$ in current realization to $U_q\slhs$ in Chevalley realization is  already  not explicit. The Miki automorphism originates in the well known Fourier transform $\Phi$ for toroidal braid group, see Lemma \ref{tor braid auto}, which maps commutative generators $\cY_i\in\widehat\B_{m+n}$ to Knizhnik-Zamolodchikov elements. The construction is as follows.

Let $A\subset \widehat \B_{m+n}$ be the subgroup generated by  $\cY_i$ and $\Phi(\cY_i)$. As mentioned above, the vertical and horizontal algebras share a copy of the finite type quantum algebra $U_q\sl_{m|n}$. Then $A (U_q\sl_{\s})$ generates $\Es$. Indeed, $\cY_i$ acting on $U_q\sl_{\s}$ generates the vertical subalgebra, while $\Phi(\cY_i)$ acting on $U_q\sl_{\s}$ generates the horizontal subalgebra, see Lemma \ref{phi diag}.

Then, by definition, the Miki automorphism $\psi$ maps 
$$\psi(b\,g)=\Phi(b)\,g \qquad (g\in U_q\sl_{\s}, b\in A).$$ 
In other words, when acting on $\Es$, the Fourier transform of the toroidal braid group is given as conjugation by the Miki automorphism, see Proposition \ref{miki phi}. From the very construction, $\psi (U_q^{ver}\slhs)=U_q^{hor}\slhs$. We also have
$\psi (U_q^{hor}\slhs)=U_q^{ver}\slhs$ since it is known  that $\Phi^2(\cY_i)$ is in the algebra generated by $\cY_i$.

\medskip

In this paper, we construct the action of toroidal braid group and the Miki automorphism for the case $m\neq n$ and $m+n>3$.

\medskip

The paper is organized as follows. In Section \ref{sl sec}, we recall the definition of the quantum affine superalgebra $U_q\slh_{m|n}$ with any choice of parity. In Section \ref{af braid sec}, we recall the action of the extended affine braid group on $U_q\slh_\bullet=\oplus_\s U_q\slhs$ given in \cite{Y}. In Section \ref{tor section}, we introduce the quantum toroidal algebra $\Es$ associated with $\gl_{m|n}$ for any choice of parity and give a few properties. In Section \ref{tor braid sec}, we construct an action of the toroidal braid group on $\E_\bullet=\oplus_\s\Es$ and the Miki automorphism. In Appendix \ref{app 1}, we give an evaluation map using an action of the braid group of $\sl_{m+n}$ on a completion of  $U_q\slh_\bullet$.

\bigskip

{\bf Acknowledgments.} This work was partially supported by a grant from the Simons Foundation \#353831. L.B. was partially supported by the CNPq-Brazil grant 210375/2014-0. 

\section{Quantum affine superalgebra \texorpdfstring{$U_q\slh_{m|n}$}{Uqslm|n}}{\label{sl sec}}
In this section, we review definitions of the quantum affine algebra corresponding to the  superalgebra $\sl_{m|n}$ and set up our notation.

\subsection{Parities and root systems}
We work over the field $\mathbb{C}$.

A superalgebra is a $\Z_2$-graded algebra $A=A_0\oplus A_1$.  Elements of $A_0$ are called even and elements of $A_1$ odd. We denote the parity of an element $v\in A_i$ by $|v|=i$, $i\in \Z_2$.

Fix $m, n \in \Z_{\geq 0}$, such that $m\neq n$ and $N=m+n\geq 3$. We always consider various indices modulo $N$.

We consider the Lie superalgebras $\sl_{m|n}$ and $\slh_{m|n}$. The set of Dynkin nodes are $I=\{1,2,\dots,N-1\}$ and $\hat I=\{0,1,\dots,N-1\}$, respectively.
It is well-known that there are different choices of the root system which lead to different Cartan matrices and different Dynkin diagrams.  Such choices are parameterized by $N$-tuples of $\pm 1$ with exactly $m$ positive coordinates. Set
$$\Smn=\{(s_1,\dots,s_N) |\  s_i\in\{-1,1\},\ \#\{i\,|\,s_i=1\}=m\}.$$
An element $\s=(s_1,\dots,s_N)\in \Smn$ is called a {\it parity sequence}. The parity sequence of the form $\s=(1,\dots,1,-1,\dots,-1)$ is called the {\it standard parity sequence}.

Given a parity sequence $\s \in \Smn$, we have the Cartan matrix $A^{\s}=(A^{\s}_{i,j})_{i,j\in I}$ and the affine Cartan matrix $\hat{A}^{\s}=(A^{\s}_{i,j})_{i,j\in \hat{I}}$, where
\begin{align}{\label{aff cartan}}
	  & A^{\s}_{i,j}=(s_i+s_{i+1})\delta_{i,j}-s_i\delta_{i,j+1}-s_j\delta_{i+1,j} & (i,j \in \hat{I}). 
\end{align}
Note that  $|\det A^\s|=|m-n|$ and $\det \hat{A}^\s=0$.

Denote by $\sl_{\s}$ the superalgebra corresponding to Cartan matrix $A^{\s}$.
Denote by $\slhs$ the superalgebra corresponding to Cartan matrix $\hat{A}^{\s}$. 
The superalgebras $\sl_{\s}$ are all isomorphic to $\sl_{m|n}$ and the superalgebras $\slhs$  to $\slh_{m|n}$. By abuse of notation we often omit the suffix $\s$ if the parity sequence is clear from the context.

Let $P_\s$ be the integral lattice with basis ${\varepsilon}{_i}$, $i\in \hat I$, and bilinear form given by
\begin{align*}
	  & \braket{{\varepsilon}{_i}}{{\varepsilon}{_j}}=s_i\delta_{i,j} & (i,j \in \hat{I}). 
\end{align*}

Let $\Delta_\s=\{{\alpha}{_i}:={\varepsilon}{_i}-{\varepsilon}{_{i+1}}|i\in I\}$ be the set of simple roots of $\sl_{\s}$, and let $Q_\s=\oplus_{i\in I} \Z\alpha_i$ be the root lattice of $\sl_{\s}$. Let also ${\delta}{}$ be the null root of $\slhs$ satisfying $\braket{{\delta}{}}{{\delta}{}}=\braket{ {\delta}{}}{ {\alpha}{_i}}=0$, $i\in I$. Set $ {\alpha}{_0}= {\delta}{}+ {\varepsilon}{_N}- {\varepsilon}{_1}$. Then, $\hat{\Delta}_\s=\{ {\alpha}{_i}|i\in \hat{I}\}$ is the set of simple roots of $\slhs$. Note that $\braket{ {\alpha}{_i}}{ {\alpha}{_j}}= A_{i,j}^{\s}$, $i,j\in \hat{I}$, and the parity of the simple root $\alpha_i$ is given by $|\alpha_i|=:|i|=(1-s_is_{i+1})/2$.

\subsection{Quantum affine superalgebra \texorpdfstring{$U_q\slh_{m|n}$}{Uqslm|n}}{\label{sl section}}

Fix $q\in \mathbb{C}^\times$ not a root of unity and let $[k]=\frac{q^k-q^{-k}}{q-q^{-1}}$, $k\in \Z.$
We also use the notation $[X,Y]_a=XY-(-1)^{|X||Y|}aYX.$ For simplicity, we write $[X,Y]_1=[X,Y]$. The bracket  $[X,Y]_a$ satisfies the following Jacobi identity
\begin{align}{\label{q der}}
	[[X,Y]_a,Z]_b=[X,[Y,Z]_c]_{abc^{-1}}+(-1)^{|Y||Z|}c[[X,Z]_{bc^{-1}},Y]_{ac^{-1}}. 
\end{align}

Let $\s$ be a parity sequence. 

In the {\it Drinfeld-Jimbo realization}, the quantum affine superalgebra $U_q\slhs$ is generated by {\it Chevalley generators} $e_i, f_i, t_i,\, i \in \hat{I} $. The parity of generators is given by $|e_i|=|f_i|=|i|=(1-s_is_{i+1})/2$, and $|t_i|=0$. 
	
The defining relations are as follows.
\begin{align*}
	&t_it_j=t_jt_i,\quad t_ie_jt_i^{-1}=q^{A^{\s}_{i,j}}e_j,\quad t_if_jt_i^{-1}=q^{-A^{\s}_{i,j}}f_j,\\
	&[e_i,f_j]=\delta_{i,j}\frac{t_i-t_i^{-1}}{q-q^{-1}},\\
	  & [e_i,e_j]=[f_i,f_j]=0                                                                                           & (A^{\s}_{i,j}=0),           \\
	  & \lb{e_i,\lb{e_i,e_{i\pm 1}}}=\lb{f_i,\lb{f_i,f_{i\pm 1}}}=0                                                     & (A^{\s}_{i,i}\neq 0),       \\
	  & \lb{e_i,\lb{e_{i+ 1},\lb{e_i,e_{i- 1}}}}=\lb{f_i,\lb{f_{i+ 1},\lb{f_i,f_{i- 1}}}}=0                             & (mn\neq 2, A^{\s}_{i,i}=0), \\
	  & \lb{e_{i+1},\lb{e_{i-1},\lb{e_{i+1},\lb{e_{i-1},e_i}}}}=\lb{e_{i-1},\lb{e_{i+1},\lb{e_{i-1},\lb{e_{i+1},e_i}}}} & (mn=2, A^{\s}_{i,i}\neq 0), \\
	  & \lb{f_{i+1},\lb{f_{i-1},\lb{f_{i+1},\lb{f_{i-1},f_i}}}}=\lb{f_{i-1},\lb{f_{i+1},\lb{f_{i-1},\lb{f_{i+1},f_i}}}} & (mn=2, A^{\s}_{i,i}\neq 0), 
\end{align*}
where $\lb{X,Y}=[X,Y]_{q^{-\braket{\beta}{\gamma}}}$ if $X$, $Y$ have weights $ \beta, \gamma \in Q_\s$, i.e., if $t_iXt_i^{-1}=q^{\braket{\alpha_i}{\beta}}$ and $t_iYt_i^{-1}=q^{\braket{\alpha_i}{\gamma}}$ for $i\in I$.

The element $t_0t_1\dots t_{N-1}$ is central.

The subalgebra of $U_q\slhs$ generated by  $e_i, f_i, t_i,\, i \in {I}$, is isomorphic to the quantum superalgebra $U_q\sl_{\s}$.

The superalgebra $U_q\slhs$ in the Drinfeld-Jimbo realization has a $\Z^N$-grading given by 
\begin{align}{\label{DJ grading}}
	  & \deg^h (x)=\left(\deg^h_0(x), \deg^h_1(x),\dots,\deg^h_{N-1}(x) \right), 
\end{align}
where 
\begin{align*}
	  & \deg^h_i(e_{j})=\delta_{i,j},\quad \deg^h_i(f_{j})=-\delta_{i,j},\quad \deg^h_i(t_{j})=0\quad  (i,j\in \hat{I}). 
\end{align*}
	
\medskip

In the {\it new Drinfeld realization}, the quantum affine superalgebra  $U_q\slhs$ is generated by {\it current generators} 
\footnote{Our generators $x^\pm_{i,r}, h_{i,r}, c^{\pm 1}$ 
	correspond to $x^\pm_{i,r}, K_\delta^{-r/2} h_{i,r}, K_\delta^{\pm 1}$ in \cite{Y}. In particular, $k_i^+(z), k_i^-(z)$ correspond to $\psi_i(K_\delta^{-1/2}z^{-1}),\phi_i(K_\delta^{-1/2}z)$ in \cite{Y}. }
$x^\pm_{i,r}, h_{i,r}$, $k^{\pm 1}_i, c^{\pm 1}$, $i \in I,\; r\in \Z'$. Here and below, we use the following convention: $r\in\Z'$ means $r\in \Z$ if $r$ is an index of a non-Cartan current generator $x^\pm_{i,r}$, and $r\in\Z'$ means $r\in \Z\setminus\{0\}$ if $r$ is an index of a Cartan current generator $h_{i,r}$.
	
The parity of generators is given by $|x^\pm_{i,r}|=|i|=(1-s_is_{i+1})/2$, and all remaining generators have parity $0$. 
	
The defining relations are as follows.
\begin{align*}
	&\text{$c$ is central},\quad k_ik_j=k_jk_i,\quad k_ix^\pm_j(z)k_i^{-1}=q^{\pm A^{\s}_{i,j}}x^\pm_j(z),\\
	&[h_{i,r},h_{j,s}]=\delta_{r+s,0}\,\frac{[rA^{\s}_{i,j}]}{r}\frac{c^r-c^{-r}}{q-q^{-1}},\\
	&[h_{i,r},x^{\pm}_j(z)]=\pm\frac{[rA^{\s}_{i,j}]}{r}c^{-(r\pm|r|)/2}z^rx^\pm_j(z),\\
	&[x^+_i(z),x^-_j(w)]=\frac{\delta_{i,j}}{q-q^{-1}}\Bigl(\delta\bigl(c\frac{w}{z}\bigr)k_i^+(w)-\delta\bigl(c\frac{z}{w}\bigr)k_i^-(z)\Bigr),\\
	  & (z-q^{\pm A^{\s}_{i,j}}w)x^\pm_i(z)x^\pm_j(w)+(-1)^{|i||j|}(w-q^{\pm A^{\s}_{i,j}}z)x^\pm_j(w)x^\pm_i(z)=0 & (A^{\s}_{i,j}\neq 0),              \\
	  & [x^\pm_i(z),x^\pm_j(w)]=0                                                                                  & (A^{\s}_{i,j}=0),                  \\
	  & \Sym_{z_1,z_2}\lb{x^\pm_i(z_1),\lb{x^\pm_i(z_2),x^\pm_{i\pm 1}(w)}}=0\,                                    & (A^{\s}_{i,i}\neq 0,\ i\pm1\in I), \\
	  & \Sym_{{z_1,z_2}}\lb{x^\pm_i(z_1),\lb{x^\pm_{i+ 1}(y),\lb{x^\pm_i(z_2),x^\pm_{i- 1}(w)}}}=0             & (A^{\s}_{i,i}=0,\ i\pm 1 \in I),   
\end{align*}
where $x^\pm_i(z)=\sum_{k\in \Z}x^\pm_{i,k}z^{-k}\,,$ $k_i^\pm(z)=k_i^{\pm 1}\exp \left(\pm (q-q^{-1})\sum_{r>0}h_{i,\pm r}z^{\mp r}\right)$.
	
The superalgebra $U_q\slhs$ in the new Drinfeld realization has a $\Z^N$-grading given by 
\begin{align}{\label{nD grading}}
	  & \deg^v (x)=\left(\deg^v_1(x),\dots,\deg^v_{N-1}(x);\deg_\delta(x) \right), 
\end{align}
where 
\begin{align*}
	  & \deg^v_i(x^\pm_{j,r})=\pm \delta_{i,j}, \quad \deg^v_i(k_{j})=\deg^v_i(h_{j,r})=\deg^v_i(c)=0 & (i,j\in I,\, r\in \Z'), \\
	  & \deg_\delta(x^\pm_{i,r})=\deg_\delta(h_{i,r})=r, \quad \deg_\delta(k_{i})=\deg_\delta(c)=0    & (i\in I,\, r\in \Z').   
\end{align*}

The isomorphism between Drinfeld-Jimbo and new Drinfeld realizations is described 
in Proposition~\ref{braid nD}.
	
For $J\subset I$, we call the subalgebra of $U_q\slhs$ generated by $a_{j,r}, k^{\pm 1}_{j}, c^{\pm 1}$, $j\in J$, $r\in\Z'$, where $a=x^+,x^-,h$,  the {\it diagram subalgebra  associated with $J$} and denote it by $U_q^J\slhs$.
Any diagram subalgebra is isomorphic to a tensor product of $U_q\slh_{k|l}$ algebras with central elements $c$ in each factor set equal to each other.
	
\section{Affine braid group}{\label{af braid sec}}
It is well known that the role of the Weyl group for simple Lie algebras is played by an appropriate braid group in the quantum setting, see \cite{L}. In this section we recall the action of extended affine braid group of type $A$ on $U_q\slh_\bullet=\bigoplus_{\s\in \Smn} U_q\slhs$. We follow \cite{Y}.

In this section, we always assume $N\geq 4$.

\subsection{Extended affine braid group of type \texorpdfstring{$A$}{A}}\label{ssbraid}
We recall the extended affine braid group of type $A$. 

Let $\B_N$ be the group generated by elements $\tau$, $T_i$, $i\in \hat{I}$, with defining relations
\begin{align}
	  & T_iT_{j}=T_jT_{i}            & (j\neq i,i\pm 1),{\label{eab1}} \\
	  & T_jT_iT_{j}=T_iT_jT_{i}      & (j=i\pm 1),{\label{eab2}}       \\
	  & \tau T_{i-1} \tau^{-1}=T_{i} & (i\in \hat{I}){\label{eab3}}.   
\end{align}
The group $\B_N$ is called the {\it extended affine braid group} of type $A$.

Alternatively, $\B_N$ can be described as the group generated by elements $\cX_i$, $T_i$, $i\in I$, with defining relations
\begin{align}
	  & T_iT_{j}=T_jT_{i}                                  & (j\neq i,i\pm 1),{\label{eab4}}    \\
	  & T_jT_iT_{j}=T_iT_jT_{i}                            & (j=i\pm 1),{\label{eab5}}          \\
	  & \cX_i\cX_j=\cX_j\cX_i                              & (i,j\in I),{\label{eab6}}          \\
	  & T_i\cX_j=\cX_jT_i                                  & (i\neq j),{\label{eab7}}           \\
	&T_1^{-1}\cX_1T_1^{-1}=\cX_2\cX_1^{-1},{\label{eab8}}\\
	&T_{N-1}^{-1}\cX_{N-1}T_{N-1}^{-1}=\cX_{N-2}\cX_{N-1}^{-1},{\label{eab9}}\\
	  & T_i^{-1}\cX_iT_i^{-1}=\cX_{i-1}\cX_{i+1}\cX_i^{-1} & (2\leq i \leq N-2).{\label{eab10}} 
\end{align}
Note that $\B_N$ is actually generated by $T_i$,  $i\in I$, and $\cX_1$.

An isomorphism $\gamma$ between the two realizations is given by
\begin{align}{\label{braideq}}
	  &\gamma: \cX_1\mapsto \tau T_{N-1}\cdots T_1, \quad T_i\mapsto T_i \qquad (i\in I). 
\end{align}

We have a surjective group homomorphism 
\begin{align}\label{pi}
	\pi: \B_N\to \mathfrak{S}_N,  \quad \tau\mapsto \tau, \quad 
	T_i\mapsto \sigma_i\qquad (i\in \hat I),\end{align}
	where we denoted $\sigma_i=(i,i+1)$, $i\in I$, $\sigma_0=(1,N)$, and, by an abuse of notation, $\tau=(1,2,\dots,N)$. 
	
	\subsection{Action  of \texorpdfstring{$\B_N$}{BN} on Drinfeld-Jimbo realization of  \texorpdfstring{$U_q\slh_\bullet$}{Uqslh.}}{\label{ssChevbraid}}
	The symmetric group $\mathfrak{S}_N$ acts naturally on $\S_{m|n}$ by permuting indices, $\sigma  \s:=(s_{{\sigma^{-1}}(1)},\dots,s_{{\sigma^{-1}}(N)})$ for all $\sigma\in \mathfrak{S}$, $\s\in \S_{m|n}$. 
	
	The extended affine braid group also acts on $\S_{m|n}$ by $T\s=\pi(T)\s$, for $T\in \B_N$, $\s\in \S_{m|n}$, see \eqref{pi}.
	
	The next proposition describes a family of isomorphisms of quantum affine superalgebras.
	
	\begin{prop}{\cite[Prop. 8.2.1]{Y}}{\label{ybraid}} We have the following.
	
	\begin{enumerate}
	\item For $i\in \hat{I}$, $\s\in \Smn$, there exists an isomorphism of superalgebras $T_{i,\s}:U_q\slhs \rightarrow U_q\slh_{\sigma_i \s}$ given on Chevalley generators by
	\begin{align*}
	&\begin{aligned}
	  & T_{i,\s}(t_i)=t_i^{-1},\qquad                                                          &   & T_{i,\s}(t_{i\pm 1})=t_it_{i\pm 1},                                     \\
	  & T_{i,\s}(e_i)=-s_i f_i t_i,\qquad                                                      &   & T_{i,\s}(f_i)=-s_{i+1}t_i^{-1} e_i,                                     \\
	  & T_{i,\s}(e_{i-1})=s_{i+1}q^{-s_{i+1}}\lb{e_{i-1},e_i},\qquad                           &   & T_{i,\s}(e_{i+1})=s_{i}q^{-s_{i}}(-1)^{|e_i||e_{i+1}|}\lb{e_{i+1},e_i}, \\
	  & T_{i,\s}(f_{i-1})=-(-1)^{|f_i||f_{i-1}|}\lb{f_{i-1},f_i},                              &   & T_{i,\s}(f_{i+1})=-\lb{f_{i+1},f_i},                                    
	\end{aligned}\\
	&\begin{aligned}
	&T_{i,\s}(e_j)=e_j,\ \qquad T_{i,\s}(f_j)=f_j,\qquad\  T_{i,\s}(t_j)=t_j\qquad    &(j\neq i, i\pm 1).
	\end{aligned}
	\end{align*}
	The parities on the r.h.s. correspond to the generators of the target algebra $U_q\slh_{\sigma_i \s}$. 
	
	\item The left-inverse of $T_{i,\s}\,,$  $(T_{i,\s})^{-1}:U_q\slh_{\sigma_i \s} \rightarrow U_q\slhs\,,$ is given by
	\begin{align*}
	&\begin{aligned}
	  & (T_{i,\s})^{-1}(t_i)=t_i^{-1},\qquad                                                   &   & (T_{i,\s})^{-1}(t_{i\pm 1})=t_it_{i\pm 1},                              \\
	  & (T_{i,\s})^{-1}(e_i)=-s_{i+1} t_i^{-1} f_i ,\qquad                                     &   & (T_{i,\s})^{-1}(f_i)=-s_{i} e_i t_i,                                    \\
	  & (T_{i,\s})^{-1}(e_{i-1})=s_{i}q^{-s_{i}}(-1)^{|e_i||e_{i-1}|}\lb{e_{i},e_{i-1}},\qquad &   & (T_{i,\s})^{-1}(e_{i+1})=s_{i+1}q^{-s_{i+1}}\lb{e_{i},e_{i+1}},         \\
	  & (T_{i,\s})^{-1}(f_{i+1})=-(-1)^{|f_i||f_{i+1}|}\lb{f_{i},f_{i+1}},                     &   & (T_{i,\s})^{-1}(f_{i-1})=-\lb{f_{i},f_{i-1}},                           
	\end{aligned}\\
	&\begin{aligned}
	&(T_{i,\s})^{-1}(e_j)=e_j,\qquad (T_{i,\s})^{-1}(f_j)=f_j,\qquad (T_{i,\s})^{-1}(t_j)=t_j\qquad    &(j\neq i, i\pm 1).
	\end{aligned}
	\end{align*}
	The parities on the r.h.s. correspond to the generators of the target algebra $U_q\slh_{\s}$.
	
	\item For $\s\in \Smn$, there exist an isomorphism of superalgebras $\tau_{\s}:U_q\slhs\rightarrow U_q\slh_{\tau \s}$ given on Chevalley generators by 
	\begin{align*}
	&\tau_{\s}(x_i)=x_{i+ 1} &(x=e,f,t).
	\end{align*}
	\end{enumerate}
	\qed
	\end{prop} 
	We note the following useful formula
	\begin{align}\label{TTX=X}
	&(T_i\, T_{i\pm 1})_\s (x_i)=x_{i\pm 1} &(x=e,f,t).
\end{align}

The isomorphisms $T_{i,\s}$ and $\tau_{\s}$ change the grading in the Drinfeld-Jimbo realization as follows.

If $\deg^h(x)=(d_0,d_1,\dots,d_{i-1}, d_i, d_{i+1},\dots,d_{N-1})$, then
\begin{equation}{\label{T deg}}
	\begin{aligned}
		  & \deg^h(T_{i,\s}(x))=(d_0,d_1,\dots,d_{i-1}, d_{i-1}+d_{i+1}-d_i, d_{i+1},\dots,d_{N-1}) \qquad (i\in\hat I), \\
		  & \deg^h(\tau_{\s}(x))=(d_{N-1},d_0,d_1,\dots,d_{N-2}).                                                        
	\end{aligned}
\end{equation}

The isomorphisms generate a groupoid if one considers
the category whose objects are the superalgebras $U_q\slhs$, $\s\in \Smn$, and whose morphisms are $\tau_\s$, $T_{i,\s}$, $i\in \hat{I}, \s \in \Smn$, their inverses, and compositions.

In our situation, the groupoid structure is equivalent to the group action as follows.

Define the following automorphisms of $U_q\slh_\bullet=\bigoplus_{\s\in \Smn} U_q\slhs$
\begin{align}{\label{autosum}}
	  & \tau=\bigoplus_{\s\in \Smn} \tau_{\s}\,, & T_i=\bigoplus_{\s\in \Smn} T_{i,\s} &   & (i\in \hat{I}). 
\end{align}
Note that, by abuse of notation, we denote by $\tau$ both the automorphism above and the element of $\mathfrak{S}_N$.
\begin{prop}{\cite[Prop. 8.2.2]{Y} \label{braidaff}}
	The automorphisms  $\tau$, $T_i$, $i\in \hat{I}$, define an action of the extended affine braid group $\B_N$ on $U_q\slh_\bullet$, i.e., they satisfy the relations \eqref{eab1}-\eqref{eab3}. \qed
\end{prop}

We adopt the following convention. For $T\in \B_N$, we denote $T_\s$ the restriction of $T$ to the $U_q\slhs$ summand in $U_q\slh_\bullet$. Note that the image of $T_\s$ is  also a particular summand in $U_q\slh_\bullet$, namely $U_q\slh_{T\s}$. For example,  $(\tau T_{i}T_{j}T_{k})_\s= \tau_{\sigma_i \sigma_j\sigma_k\s} T_{i,\sigma_j\sigma_k\s}T_{j,\sigma_k\s}T_{k,\s}$ is mapping $U_q \slhs$ to $U_q\slh_{\tau\sigma_i \sigma_j\sigma_k\s}$. We use a similar convention with other maps, see, for example, Theorem \ref{Tbraid} below.

Note that the action of $\B_N$ on $\Smn$ is transitive. In particular, Proposition \ref{ybraid} implies that all superalgebras $U_q\slhs$, $\s \in \Smn$, are isomorphic.

\subsection{Action of \texorpdfstring{$\B_N$}{BN} on new Drinfeld realization of \texorpdfstring{$U_q\slh_\bullet$}{Uqsl.}}{\label{braid nD}}

We have an action of extended affine braid group $\B_N$ on $U_q\slh_\bullet$ given in Chevalley generators, see Section \ref{ssChevbraid}. The group $\B_N$ contains elements $\cX_i$, $i\in I$, see Section \ref{ssbraid}. The elements $\cX_i$ preserve the parity, $\cX_i\s=\s$, for all $\s\in \Smn$, and, therefore, $(\cX_{i})_\s$ is an automorphism of $U_q\slhs$. These automorphisms are used to obtain an isomorphism between the two different realizations of $U_q\slhs$ (similar to the even case, see \cite{B}). 

\begin{prop}{\cite[Theorem 8.5.1]{Y}}\label{X} There exists an isomorphism $\iota_\s$ from the new Drinfeld to the Drinfeld-Jimbo realization of $U_q\slhs$ mapping:
	\begin{align}{\label{nDDJ}}
		  & x^+_{i,r}\mapsto (-1)^{ir}\cX_{i,\s}^{-r} (e_i), &   & x^-_{i,r}\mapsto (-1)^{ir}\cX_{i,\s}^r (f_i), &   & k_i\mapsto t_i, &   & c\mapsto t_0t_1\cdots t_{N-1} &   & (r\in \Z,i\in I). 
	\end{align}
	\qed
\end{prop}
The identifications $\iota_\s$ allow us to study action of $\B_N$ on the new Drinfeld realization. One can describe action of the $\cX_{i,\s}$ in current generators explicitly.

\begin{prop}\label{X action} For $i\in I$, $\s\in \Smn$, the action of  $\cX_{i,\s}$ in current generators is given by
	\begin{align*}
		  & \cX_{i,\s}(x^\pm_{j,r})=(-1)^{i\delta_{ij}}x^\pm_{j,r\mp \delta_{ij}}, &   & \cX_{i,\s}(k_j)=c^{-\delta_{ij}}k_j, &   &                     \\
		  & \cX_{i,\s}(h_{j,r})=h_{j,r},                                           &   & \cX_{i,\s}(c)=c                      &   & (r\in \Z', j\in I). 
	\end{align*}
	\begin{proof}
	The above equalities with $i=j$ follow from \eqref{nDDJ}. 
	If $i\neq j$, by \cite[Prop. 8.2.3]{Y}, we have $\cX_{i,\s}(a_j)=a_j$, $(a=e, f, t)$, and the proposition follows from the commutativity \eqref{eab6} of the operators $\cX_{i,\s}$.  
	\end{proof}
\end{prop}

For the action of $T_{i,\s}$ in current generators we have some partial information.

\begin{lem}{\label{T cur}} 
	For $i\in I$,  we have
	\begin{align}
		T_{i,\s}(a_{j,r})=a_{j,r} \qquad (r\in \Z',\ j\in I,\ i\neq j,j\pm1,\  a=x^+, x^-,h).     \label{Tfix1} 
	\end{align}
	Moreover,
	\begin{align}
		  & T_{i,\s}(x^+_{i+1,r})=s_{i}q^{-s_{i}}(-1)^{|i||i+1|}\lb{x^+_{i+1,r},x^+_{i,0}} \qquad & (r\in \Z),\label{Tfix2} \\
		  & T_{i,\s}(x^+_{i-1,r})=s_{i+1}q^{-s_{i+1}}\lb{x^+_{i-1,r},x^+_{i,0}} \qquad            & (r\in \Z),\label{Tfix3} \\  
		  & T_{i,\s}(x^-_{i+1,r})=-\lb{x^-_{i+1,r},x^-_{i,0}} \qquad                              & (r\in \Z),\label{Tfix4} \\
		  & T_{i,\s}(x^-_{i-1,r})=-(-1)^{|i||i-1|}\lb{x^-_{i-1,r},x^-_{i,0}} \qquad               & (r\in \Z).\label{Tfix5} 
	\end{align}
	The parities on the r.h.s. correspond to the generators of target algebra $U_q\slh_{\sigma_i\s}$.
	
	We also have, $T_{i,\s} U^{\{i\}}_q\slhs\subset U^{\{i\}}_q\slh_{\sigma_i\s}$ if $i\neq 1,N-1$. 
	
	Finally, $T_{1,\s} U^{\{1\}}_q\slhs\subset U^{\{1,2\}}_q\slh_{\sigma_1\s}$ and $T_{N-1,\s} U^{\{N-1\}}_q\slhs\subset U^{\{N-1,N-2\}}_q\slh_{\sigma_{N-1}\s}$.
\end{lem}
\begin{proof}
	Equations \eqref{Tfix1}-\eqref{Tfix5} follow from relation \eqref{eab7} and Proposition \ref{ybraid}.
	
	To prove the last part, we note that if $i\neq 1$ then $x^\pm_{i,0}$ and $h_{i-1,\pm1}$ generate a subalgebra containing $(U^{\{i\}}_q\slhs)^\pm$. Here and below we denote by suffix  $+$ (resp. $-$) the subalgebras generated by non-negative (resp. non-positive) modes of the generating currents in Drinfeld realization. Note that such subalgebras are $\Z_{\geq 0}$-graded with finite-dimensional graded components. Therefore, $T_{i,\s} (U^{\{i\}}_q\slhs)^\pm\subset (U^{\{i,i-1\}}_q\slh_{\sigma_i\s})^\pm$. Similarly, if $i\neq N-1$, we have $T_{i,\s} (U^{\{i\}}_q\slhs)^\pm\subset (U^{\{i,i+1\}}_q\slh_{\sigma_i\s})^\pm$. Now, from the PBW theorem \cite[Theorem 5.7]{T2}, we have the intersection  $(U^{\{i,i-1\}}_q)^\pm\cap (U^{\{i,i+1\}}_q)^\pm= (U^{\{i\}}_q)^\pm$. The lemma follows.
\end{proof}

We can also write the inverse of the isomorphism $\iota_\s$.

\begin{lem}\label{iota inverse} The isomorphism $\iota^{-1}_\s$ maps
	\begin{align}
		  & e_i\mapsto x^+_{i,0},\quad f_i\mapsto x^-_{i,0},\quad t_i\mapsto k_i & (i \in I),\label{DJnDi} \\
		&t_0\mapsto c(k_1k_2\cdots k_{m+n-1})^{-1},\label{DJnDt0}\\
		&e_0\mapsto \left( \cX_1T_{N-1}\cdots T_2 T_1^{-1}\right)_{\tau\s}(x^+_{1,0}), \label{DJnDe} \\
		&f_0\mapsto \left( \cX_1T_{N-1}\cdots T_2 T_1^{-1}\right)_{\tau\s}(x^-_{1,0}).\label{DJnDf}
	\end{align}
\end{lem}
\begin{proof} Equations \eqref{DJnDi} and \eqref{DJnDt0} are clear.

We check \eqref{DJnDe}, the check for equation \eqref{DJnDf} is analogous. By equation \eqref{braideq} we have
$$
e_0=(\tau^{-1})_{\tau\s}\ (e_1)=\left( T_{N-1}\cdots T_2 T_1\cX_1^{-1}\right)_{\tau\s}(x^+_{1,0}),
$$
using equations \eqref{eab7} and \eqref{eab8}, and noting that  $\cX_2$ commutes with $T_1$ and acts trivially on $x^+_{1,0}$, we have
$$
\left( T_{N-1}\cdots T_2 T_1\cX_1^{-1}\right)_{\tau\s}(x^+_{1,0})=\left( T_{N-1}\cdots T_2 \cX_1\cX_2^{-1}T_1^{-1}\right)_{\tau\s}(x^+_{1,0})=\left( \cX_1T_{N-1}\cdots T_2 T_1^{-1}\right)_{\tau\s}(x^+_{1,0}).
$$
\end{proof}
Note that in Lemma \ref{iota inverse} we apply $T_i$ only to Chevalley generators, therefore the formulas are explicit.

The correspondence between the $\Z^N$-grading in the two realizations of $U_q\slhs$ is as follows.

If $x\in U_q\slhs$ is given in the new Drinfeld realization and $\deg^v(x)=(d_1^v,\dots,d_{N-1}^v;d_\delta)$, then the grading in the Drinfeld-Jimbo realization is
\begin{align}
	  & \deg^h(\iota_{\s}(x))=(d_\delta,d_1^v+d_\delta,\dots,d_{N-1}^v+d_\delta).\label{iota deg} 
\end{align}

Similarly, if $x\in U_q\slhs$ is given in the Drinfeld-Jimbo realization and $\deg^h(x)=(d_0^h,d_1^h,\dots,d_{N-1}^h)$, then the grading in the new Drinfeld realization is
\begin{align}
	  & \deg^v(\iota_{\s}^{-1}(x))=(d_1^h-d_0^h,\dots,d_{N-1}^h-d_0^h;d_0^h).\label{iota^-1 deg} 
\end{align}

\medskip

\subsection{The anti-automorphisms \texorpdfstring{$\varphi$}{phi} and \texorpdfstring{$\eta$}{eta}}
We have two anti-automorphisms of $U_q\slhs$ which will be used in Sections \ref{tor section} and \ref{tor braid sec}.

\begin{lem}{\label{philem}}
There exists a superalgebra anti-automorphism $\varphi:\:U_q\slh_\bullet\rightarrow U_q\slh_\bullet,$
where $\varphi=\oplus_{\s \in \Smn}\varphi_{\s}$, and
	for $\s \in \Smn$, the anti-automorphism $\varphi_\s:U_q\slh_{\s}\rightarrow U_q\slhs$ is given on Chevalley generators by
	\begin{align*}
		  & \varphi_\s (e_i)=e_i, &   & \varphi_\s (f_i)= f_i, &   & \varphi_\s ( t_i)= t_i^{-1} &   & (i\in \hat I). 
	\end{align*}
	Moreover, 
	\begin{align*}
		  & (\varphi\, T_{i}\, \varphi)_\s=(T_{i,\sigma_{i}\s})^{-1} & (i\in \hat I). 
	\end{align*}
	\begin{proof}
		This is checked by a straightforward computation.
	\end{proof}
\end{lem}

Note that $\varphi_\s$ preserves the grading \eqref{DJ grading}.
\bigskip

\begin{lem}{\label{etalem}}
There exists a superalgebra anti-automorphism $\eta:\:U_q\slh_\bullet\rightarrow U_q\slh_\bullet,$
where $\eta=\oplus_{\s \in \Smn}\eta_{\s}$, and
	for $\s \in \Smn$, the anti-automorphism
	$\eta_\s:U_q\slhs \rightarrow U_q\slhs$  is given on current generators by
	\begin{align*}
		  & \eta_\s(c)=c, &   & \eta_\s(k_i^{\pm}(z))=k_i^{\mp}(cz^{-1}), &   & \eta_\s(x^{\pm}_i(z))= x^{\pm}_i(z^{-1}) &   & (i\in I). 
	\end{align*}
	Moreover, 
	\begin{align*}
		  & (\eta\, T_{i}\, \eta)_\s=(T_{i,\sigma_{i}\s})^{-1} & (i\in I). 
	\end{align*}
	\begin{proof}
		The existence of this anti-automorphism is checked directly. 
		
		For the last equality, note that $\eta_\s$ coincides with $\varphi_\s$ on the subalgebra generated by $x^\pm_{i,0}, k_i, c$, $i\in I$. Also, by the isomorphism between the new Drinfeld and Drinfeld--Jimbo realizations of $U_q\slhs$, it suffices only to check the identity on $x^\pm_{1,\mp 1}$.
		
		We verify 
		$(\eta\, T_{i}\, \eta)_\s =(T_{i,\sigma_{i}\s})^{-1}$ on $x^-_{1,1}$ for $i=1,2$. The remaining values of $i$ are trivial. The check for $x^+_{1,-1}$ is analogous.
		
		Using the relation \eqref{eab8} we have
		\begin{align*}
			  & (\eta T_1 \eta)_\s (x^-_{1,1})=-(\eta T_{1} \cX_{1}^{-1})_\s(x^-_{1,0}) =-(\eta \cX_1 \cX_{2}^{-1}T_{1}^{-1})_{\s}(x^-_{1,0})=\eta_{\sigma_1 \s}(-s_1c^{-1}x^+_{1,-1}k_1)=-s_1c^{-1}k_1^{-1}x^+_{1,1}, \\
			  & (T_{1,\sigma_{1}\s})^{-1}(x^-_{1,1})=-(T_{1}^{-1}\cX_1)_{\s}(x^-_{1,0})=-(\cX_2 \cX_1^{-1}T_1)_{\s}(x^-_{1,0})=-s_1c^{-1}k_1^{-1}x^+_{1,1}.                                                            
		\end{align*}
		
		And using the relation \eqref{eab7} we have
		\begin{align*}
			  & (\eta T_2 \eta)_\s (x^-_{1,1})=-(\eta T_{2} \cX_{1}^{-1})_{\s}(x^-_{1,0}) =(\eta \cX_{1}^{-1})_{\sigma_2  \s}((-1)^{|1||2|}\lb{x^-_{1,0},x^-_{2,0}})=-\lb{x^-_{2,0},x^-_{1,1}}, \\
			  & (T_{2,\sigma_{2}\s})^{-1}(x^-_{1,1})=-(T_{2}^{-1}\cX_1)_{\s}(x^-_{1,0})=\cX_{1,\sigma_2 \s}(\lb{x^-_{2,0},x^-_{1,0}})=-\lb{x^-_{2,0},x^-_{1,1}}.                                
		\end{align*}
		This completes the proof.
	\end{proof}
\end{lem}

Note that, if $x\in U_q\slhs$ is given in the new Drinfeld realization and $\deg^v(x)=(d_1,\dots,d_{N-1};d_\delta)$, then 
\begin{equation}{\label{eta deg}}
	\begin{aligned}
		  & \deg^v(\eta_{\s}(x))=(d_1,\dots,d_{N-1};-d_\delta). 
	\end{aligned}
\end{equation}

Both anti-automorphisms $\varphi_\s$ and $\eta_\s$ are anti-involutions: $\varphi_\s^2=\eta_\s^2=1$.

\section{Quantum toroidal superalgebra \texorpdfstring{$\Es$}{Es}}{\label{tor section}}
The quantum toroidal algebra associated with $\gl_{m|n}$ and standard parity was introduced in \cite{BM}. In this section, we introduce the quantum toroidal algebra $\E_\s$ associated with $\gl_{m|n}$ for any choice of parity $\s$. We give a few properties of these algebras.

\subsection{Definition of \texorpdfstring{$\Es$}{Es}}
Fix $d,q \in \mathbb{C}^\times$ and define 
$$q_1=d\,q^{-1},\quad q_2=q^2,\quad q_3=d^{-1}q^{-1}.$$

Note that $q_1q_2q_3=1$. In this paper we always assume that $q_1,q_2$ are generic, meaning that $q_1^{n_1}q_2^{n_2}q_3^{n_3}=1$, $n_1, n_2,n_3\in \mathbb{Z}$, iff  $n_1=n_2=n_3$. Fix also $d^{1/2}, q^{1/2}\in \mathbb{C}^\times$ such that $(d^{1/2})^2=d,\; (q^{1/2})^2=q$.

Recall the affine Cartan matrix $\hat A^\s=(A_{i,j}^\s)_{i,j,\in \hat I}$, see \eqref{aff cartan}.

We also define the matrix ${M}^{\s}=({M}^{\s}_{i,j})_{i,j\in \hat{I}}$ by $M^{\s}_{i+1,i}=-M^{\s}_{i,i+1}=s_{i+1}$, and $M^{\s}_{i,j}=0$, $i\neq j\pm 1$.

\begin{dfn}{\label{defE}}
	The \textit{quantum toroidal algebra associated with $\gl_{m|n}$} and a parity sequence $\s$ is the unital associative superalgebra $\Es=\Es(q_1,q_2,q_3)$  generated by $E_{i,r},F_{i,r},H_{i,r}$, 
	and invertible elements $K_i$, $C$, where $i\in \hat{I}$, $r\in \Z'$, subject to the defining relations \eqref{relCK}-\eqref{Serre6} below.
	The parity of the generators is given by $|E_{i,r}|=|F_{i,r}|=|i|=(1-s_is_{i+1})/2$, and all remaining generators have parity $0$. 
\end{dfn}

We use generating series
\begin{align*}
	E_i(z)=\sum_{k\in\Z}E_{i,k}z^{-k}, \quad                                            
	F_i(z)=\sum_{k\in\Z}F_{i,k}z^{-k}, \quad                                            
	K^{\pm}_i(z)=K_i^{\pm1}\exp\Bigl(\pm(q-q^{-1})\sum_{r>0}H_{i,\pm r}z^{\mp r}\Bigr). 
\end{align*}
	
\medskip
Let also $\displaystyle{\delta\left(z\right)=\sum_{n\in \mathbb{Z}} z^n}$ be the formal delta function.
\medskip
	
Then the defining relations are as follows.
\bigskip

\noindent{\bf $C,K$ relations}
\begin{align}{\label{relCK}}
	&\text{$C$ is central}, 
	  &   & K_iK_j=K_jK_i,                             
	  &   & K_iE_j(z)K_i^{-1}=q^{A_{i,j}^{\s}}E_j(z),  
	  &   & K_iF_j(z)K_i^{-1}=q^{-A_{i,j}^{\s}}F_j(z). 
\end{align}
\bigskip
	
\noindent{\bf $K$-$K$, $K$-$E$ and $K$-$F$ relations}
\begin{align}
	  & K^\pm_i(z)K^\pm_j (w) = K^\pm_j(w)K^\pm_i (z),                                                                                              
	\label{KK1}\\
	  & \frac{d^{M_{i,j}^{\s}}C^{-1}z-q^{A_{i,j}^{\s}}w}{d^{M_{i,j}^{\s}}Cz-q^{A_{i,j}^{\s}}w}                                                      
	K^-_i(z)K^+_j (w) 
	=
	\frac{d^{M_{i,j}^{\s}}q^{A_{i,j}^{\s}}C^{-1}z-w}{d^{M_{i,j}^{\s}}q^{A_{i,j}^{\s}}Cz-w}
	K^+_j(w)K^-_i (z),
	\label{KK2}\\
	  & (d^{M_{i,j}^{\s}}z-q^{A_{i,j}^{\s}}w)K_i^\pm(C^{-(1\pm1)/2}z)E_j(w)=(d^{M_{i,j}^{\s}}q^{A_{i,j}^{\s}}z-w)E_j(w)K_i^\pm(C^{-(1\pm1) /2}z),   
	\label{KE}\\
	  & (d^{M_{i,j}^{\s}}z-q^{-A_{i,j}^{\s}}w)K_i^\pm(C^{-(1\mp1)/2}z)F_j(w)=(d^{M_{i,j}^{\s}}q^{-A_{i,j}^{\s}}z-w)F_j(w)K_i^\pm(C^{-(1\mp1) /2}z). 
	\label{KF}
\end{align}
\bigskip

\noindent{\bf $E$-$F$ relations}
\begin{align}\label{EF}
	  & [E_i(z),F_j(w)]=\frac{\delta_{i,j}}{q-q^{-1}} 
	(\delta\left(C\frac{w}{z}\right)K_i^+(w)
	-\delta\left(C\frac{z}{w}\right)K_i^-(z)).
\end{align}
\bigskip
	
\noindent{\bf $E$-$E$ and $F$-$F$ relations}
\begin{align}
	  & [E_i(z),E_j(w)]=0\,, \quad [F_i(z),F_j(w)]=0\, \quad                                                                     & (A_{i,j}^{\s}=0),   \label{EE FF} 
	\\
	  & (d^{M_{i,j}^{\s}}z-q^{A_{i,j}^{\s}}w)E_i(z)E_j(w)=(-1)^{|i||j|}(d^{M_{i,j}^{\s}}q^{A_{i,j}^{\s}}z-w)E_j(w)E_i(z) \quad   & (A_{i,j}^{\s}\neq0),              
	\\
	  & (d^{M_{i,j}^{\s}}z-q^{-A_{i,j}^{\s}}w)F_i(z)F_j(w)=(-1)^{|i||j|}(d^{M_{i,j}^{\s}}q^{-A_{i,j}^{\s}}z-w)F_j(w)F_i(z) \quad & (A_{i,j}^{\s}\neq0).              
\end{align}
\bigskip
	
\noindent{\bf Serre relations}
\begin{align}
	  & \Sym_{{z_1,z_2}}\lb{E_i(z_1),\lb{E_i(z_2),E_{i\pm1}(w)}}=0\,\quad & (A^{\s}_{i,i}\neq 0),\label{Serre1} \\ 
	  & \Sym_{{z_1,z_2}}\lb{F_i(z_1),\lb{F_i(z_2),F_{i\pm1}(w)}}=0\,\quad & (A^{\s}_{i,i}\neq 0),               
	\label{Serre2}
\end{align}
	
If $mn\neq 2$,
\begin{align}
	  & \Sym_{{z_1,z_2}}\lb{E_i(z_1),\lb{E_{i+1}(y),\lb{E_i(z_2),E_{i-1}(w)}}}=0\,\quad & (A^{\s}_{i,i}= 0),\label{Serre3} 
	\\
	  & \Sym_{{z_1,z_2}}\lb{F_i(z_1),\lb{F_{i+1}(y),\lb{F_i(z_2),F_{i-1}(w)}}}=0\,\quad & (A^{\s}_{i,i}= 0).\label{Serre4} 
\end{align}
	
If $mn=2$,
\begin{align}
	\label{Serre5} & \Sym_{{z_1,z_2}}\Sym_{{w_1,w_2}}\lb{E_{i-1}(z_1),\lb{E_{i+1}(w_1),\lb{E_{{i-1}}(z_2),\lb{E_{i+1}(w_2),E_{i}(y)}}}}=    \\ \notag
	  & =\Sym_{{z_1,z_2}}\Sym_{{w_1,w_2}}\lb{E_{i+1}(w_1),\lb{E_{i-1}(z_1),\lb{E_{{i+1}}(w_2),\lb{E_{i-1}(z_2),E_{i}(y)}}}}\, &   & (A^{\s}_{i,i}\neq  0), \\
	\label{Serre6} & \Sym_{{z_1,z_2}}\Sym_{{w_1,w_2}}\lb{F_{i-1}(z_1),\lb{F_{i+1}(w_1),\lb{F_{{i-1}}(z_2),\lb{F_{i+1}(w_2),F_{i}(y)}}}}=    \\ \notag
	  & =\Sym_{{z_1,z_2}}\Sym_{{w_1,w_2}}\lb{F_{i+1}(w_1),\lb{F_{i-1}(z_1),\lb{F_{{i+1}}(w_2),\lb{F_{i-1}(z_2),F_{i}(y)}}}}\, &   & (A^{\s}_{i,i}\neq  0). 
\end{align}

\medskip
	
The relations \eqref{KK1}-\eqref{KF} are equivalent to 
\begin{align}
	  & [H_{i,r},E_j(z)]= \frac{[r A^\s_{i,j}]}{r}d^{-rM^\s_{i,j}}C^{-(r+|r|)/2}           
	\,z^r E_j(z)\,,\label{HE}
	\\
	  & [H_{i,r},F_j(z)]=- \frac{[r A^\s_{i,j}]}{r}d^{-rM^\s_{i,j}}C^{-(r-|r|)/2}          
	\,z^r F_j(z)\,,\label{HF}
	\\
	  & [H_{i,r},H_{j,s}]=\delta_{r+s,0} \cdot  \frac{[r A^\s_{i,j}]}{r}d^{-rM^\s_{i,j}}\, 
	\frac{C^r-C^{-r}}{q-q^{-1}}\,,\label{h_rel}
\end{align}
for all $r\in \Z'$, $i,j \in \hat{I}$.

The relations \eqref{Serre1} and \eqref{Serre2} are also satisfied if $A^{\s}_{i,i}= 0$, due to the quadratic relations \eqref{EE FF}.

The element	$K:=K_0 K_1\cdots K_{N-1}$ is central.

\medskip	
	
For any $J\subset \hat I$, let $\Es^J\subset \Es$ be the subalgebra generated by $E_i(z), F_i(z), K_i^\pm(z)$, $C$, $i\in J$. We call $\Es^J$ the {\it diagram subalgebra associated with J} of $\E_\s$.

\subsection{Some properties of \texorpdfstring{$\E_\s$}{Es}}

For each $i \in \hat{I}$, the superalgebra $\Es$ has a $\Z$-grading given by 
\begin{align*}
	  & \deg_i(E_{j,r})=\delta_{i,j},\quad \deg_i(F_{j,r})=-\delta_{i,j},\quad \deg_i(H_{j,r})=\deg_i  (K_j)=\deg_i(C)=0\quad  (j\in \hat{I},\, r\in \Z'). 
\end{align*}

There is also the \textit{homogeneous $\Z$-grading} given by
\begin{align*}
	\deg_{\delta}(E_{j,r})=\deg_{\delta}(F_{j,r})=r, \quad \deg_{\delta}(H_{j,r})=r, \quad \deg_{\delta}(K_j)=\deg_{\delta}(C)=0 \quad  (j\in \hat{I},\, r\in \Z'). 
\end{align*}
	
Thus, the superalgebra $\Es$ has a $\Z^{N+1}$-grading given on a homogeneous element $X\in \Es$ by
\begin{align}
	\deg (X)=\left(\deg_0(X), \deg_1(X),\dots,\deg_{N-1}(X);\deg_{\delta}(X) \right). 
\end{align}

\medskip 

The superalgebra $\Es$ has a graded topological Hopf superalgebra structure given on generators by
\begin{align*}
	  & \Delta E_i(z)=E_i(z)\otimes 1 + K^-_i(z)\otimes E_i(C_1z),                                 \\
	  & \Delta F_i(z)=F_i(C_2z)\otimes K^+_i(z) + 1\otimes F_i(z),                                 \\
	  & \Delta K^+_i(z)=K^+_i(C_2z)\otimes K^+_i(z),                                               \\
	  & \Delta K^-_i(z)=K^-_i(z)\otimes K^-_i(C_1z),                                               \\
	  & \Delta C=C\otimes C,                                                                       \\
	  & \varepsilon(E_i(z))=\varepsilon(F_i(z))=0,\quad \varepsilon(K^\pm _i(z))=\varepsilon(C)=1, \\
	  & S(E_i(z))=-\left(K^-_i(C^{-1}z)\right)^{-1}E_i(C^{-1}z),                                   \\
	  & S(F_i(z))=-F_i(C^{-1}z)\left(K^+_i(C^{-1}z)\right)^{-1},                                   \\
	  & S(K^\pm_i(z))=\left(K^\pm _i(C^{-1}z) \right)^{-1},\quad S(C)=C^{-1},                      
\end{align*}
where $C_1=C\otimes 1$, $C_2=1\otimes C$. In particular, we have $S(xy)=(-1)^{|x||y|}S(y)S(x)$. Note that we always use the graded tensor product multiplication defined for homogeneous elements $x_1, x_2, y_1, y_2 \in\Es $ by $(x_1\otimes y_1)(x_2\otimes y_2)=(-1)^{|y_1||x_2|}x_1x_2\otimes y_1y_2$ and extended to the whole algebra by linearity.

\subsection{Horizontal and vertical subalgebras}{\label{hv sec}}
Let $\s$ be a parity sequence. For $i\in I$, define $\mu_\s(i)=-\sum_{j=1}^i s_j$. 
Define the {\it vertical homomorphism} of superalgebras $v_\s:\U_q\slhs \rightarrow \Es$ by
\begin{align*}
	  & v_\s(x^+_i(z))= E_i(d^{\mu_\s(i)}z), \quad  v_\s(x^-_i(z))= F_i(d^{\mu_\s(i)}z),\quad v_\s(k^\pm_i(z))= K_i^\pm(d^{\mu_\s(i)}z),\quad   v_\s(c)= C & (i\in I). 
\end{align*}

Note that if $x \in \U_q\slhs$ and $\deg^v(x)=(d_1,d_2,\dots,d_{N-1};d_\delta)$, then 
\begin{equation}{\label{v deg}}
	\begin{aligned}
		  & \deg(v_{\s}(x))=(0,d_1,d_2,\dots,d_{N-1};d_\delta). 
	\end{aligned}
\end{equation}

\begin{prop}{\label{lemv}}
	The vertical homomorphism $v_\s$ is injective for generic values of parameters.
\end{prop}	
\begin{proof}
	This proposition was proved in \cite{BM} in the case $\s$ is the standard parity sequence using the existence of the evaluation map. An evaluation map for any choice of parity $\s$ under the resonance condition $q_3^{m-n}=C^2$ is given in Appendix \ref{app 1}. It provides a left inverse for $v_\s$ and therefore, $v_\s$ is injective under the resonance condition. The property of being injective is open. Thus $v_\s$ is injective for generic parameters.
\end{proof}
The image of the vertical homomorphism coincides with $\E^I_\s$. We denote this subalgebra $U^{ver}_q\slhs$ and call it the {\it vertical quantum affine $\mathfrak{sl}_{\s}$}.

The vertical subalgebra $U^{ver}_q\slhs$ is a Hopf subalgebra of $\Es$.

\begin{cor}
	Let $J\subset \hat I$, $J\neq \hat I$. Then for generic values of parameters the diagram subalgebra $\E_\s^J$ is isomorphic to tensor product of quantum affine superalgebras $U_q\slh_{k|l}$ with central elements $c$ in all factors set equal to each other. \qed
\end{cor}

We denote by $U^{hor}_q\slhs$ the subalgebra of $\Es$ generated by $E_{i,0}, F_{i,0}, K_i, i \in \hat{I}$, and we call it the {\it horizontal quantum affine $\mathfrak{sl}_{\s}$}.
	
We have a {\it horizontal homomorphism} of superalgebras $h_\s:\U_q\slhs \rightarrow \Es$ given by
\begin{align*}
	  & e_i\mapsto E_{i,0}\,, \quad  f_i\mapsto F_{i,0}\,,\quad t_i\mapsto K_i & (i\in \hat{I}), 
\end{align*}    
with image $U^{hor}_q\slhs$. 
	
We will later prove (for $N>3$) that for generic values of parameters the horizontal homomorphism $h_\s$ is injective, see Corollary \ref{h inj}. Note that it is not a Hopf algebra map.

Note that, if $x \in \U_q\slhs$ and $\deg^h(x)=(d_0,d_1,d_2,\dots,d_{N-1})$, then 
\begin{equation}{\label{h deg}}
	\begin{aligned}
		  & \deg(h_{\s}(x))=(d_0,d_1,d_2,\dots,d_{N-1};0). 
	\end{aligned}
\end{equation}

\begin{lem}
	The quantum toroidal algebra $\E_\s$ is generated by the vertical and horizontal subalgebras $U^{ver}_q\slhs$ and $U^{hor}_q\slhs$.
\end{lem}
\begin{proof}
	The only generators which are not generators of either the vertical or horizontal subalgebras are $E_{0,r}$, $F_{0,r}$, $r\in \Z^\times$. These generators are obtained as commutators of $E_{0,0}, F_{0,0}\in U^{hor}_q\slhs$ and $H_{1,\pm 1}\in U^{ver}_q\slhs$, see \eqref{HE}, \eqref{HF}.
\end{proof}

We will often use a shortcut notation 
\begin{equation}
	\begin{aligned}
		&X^+_0(z):=E_0(z), \qquad X^-_0(z):=F_0(z), \qquad \tilde{K}^\pm_0(z):=K^\pm_0(z), \\
		  & X^\pm_i(z):=v_\s(x^\pm_i(z)), \qquad \tilde{K}^\pm_i(z):=v_\s(k^\pm_i(z)) &   & (i\in I).\label{X not} 
	\end{aligned}
\end{equation}

\subsection{Morphisms}\label{sec mor}
We list some symmetries of the superalgebras $\Es$.

Set 
$$
\E_\bullet=\bigoplus_{\s\in \Smn} \Es.
$$
We have a number of maps which depend on a parity sequence $\s$. We always consider the direct sums of such maps. For example,
$h=\oplus_{\s \in \Smn} h_{\s}$ and $v=\oplus_{\s \in \Smn} v_{\s}$ are maps from $U_q\slh_\bullet$ to $\E_\bullet$.

Given $\s\in \Smn$, let $\s'=(s_{m-1},s_{m-2},\dots, s_{-n})$. The {\it diagram isomorphism} $\omega_\s:\Es(q_1,q_2,q_3) \rightarrow \E_{\s'}(q_3,q_2,q_1)$ is defined by
\begin{align*}
	  & \omega_{\s}(C)=C, &   & \omega_{\s}(A_i(z))=A_{m-i}(z) &   & (i\in \hat{I},\, A=K^\pm,E,F). 
\end{align*}
Note that it changes $d$ to $d^{-1}$.

\medskip
Given $\s\in \Smn$, let $-\s=(-s_{1},-s_{2},\dots, -s_{N})\in \S_{n|m}$. The {\it change of parity} isomorphism $\nu_\s:\Es(q_1,q_2,q_3) \rightarrow \E_{-\s}(q_3^{-1},q_2^{-1},q_1^{-1})$ is defined by
\begin{align*}
	  & \nu_\s(E_i(z))=E_{-i}(z),\quad  \nu_\s(F_i(z))=F_{-i}(z),\quad  \nu_\s(K_i^\pm(z))=-K_{-i}^\pm(z), \quad \nu_\s(C)=C\quad & (i\in \hat{I}). 
\end{align*}
Note that it changes $q$ to $q^{-1}$. 
	
\medskip

For $u\in \mathbb{C}^\times$, the {\it shift of spectral parameter} $\gamma_{u,\s}:\Es\rightarrow \Es$	is an isomorphism defined by
\begin{align*}
	  & \gamma_{u,\s}(C)=C, &   & \gamma_{u,\s}(A_i(z))=A_i(uz) &   & (i\in \hat{I},\, A=K^\pm,E,F). 
\end{align*}

\medskip

For $\s\in \Smn$, there exists an isomorphism  $\hTa_{\s}:\Es \rightarrow \E_{\tau \s}$ given by 
\begin{align}{\label{htau}}
	  & \hTa_{\s}(C)=C, &   & \hTa_{\s}(A_{i} (z))=A_{i+1} (-d^{-\s_{N}}z) &   & (i\in \hat{I},\, A=K^\pm,E,F). 
\end{align}

In the notation \eqref{X not}, the map $\hTa_\s$ takes the form
\begin{align*}
	  & \hTa_{\s}(C)=C, &   & \hTa_{\s}(A_{i} (z))=A_{i+1} (-d^{(n-m)\delta_{i,N-1}}z) &   & (i\in \hat{I},\, A=\tilde K^\pm,X^+,X^-). 
\end{align*}

\begin{prop}{\label{htauprop}} 
	The isomorphisms $\hTa_{\s}$,  $\s\in \Smn$, satisfy
	\begin{align}
		&(\hTa\,h)_\s=(h\, \tau)_\s, \label{tauh}\\
		  & (\hTa\,v)_\s(a_{i} (z))=v_{\tau \s}(a_{i+1} (-z))                        &   & (i\in I\setminus \{N-1\},\, a=k^\pm,x^\pm).\label{tauv}     \\
		  & (\hTa\,v\,T_{i})_\s\, (a_{j}(z))=(v\,T_{i+1})_{\tau \s}\, (a_{j+1} (-z)) &   & (i,j\in I\setminus \{N-1\},\, a=k^\pm,x^\pm).\label{tau Tv} 
	\end{align}
	The maps $\hTa_{\s}$ preserve the homogeneous grading and $\deg_i(\hTa_\s(X))=\deg_{i-1}(X)$, $i\in \hat I$.
	        
	\begin{proof}
		A straightforward computation shows that $\hTa_{\s}$ preserve the  homogeneous grading, satisfy equality \eqref{tauh}, and  $\deg_i(\hTa_\s(X))=\deg_{i-1}(X)$, $i\in \hat I$, if $X\in \Es$ is homogeneous.
		        
		We check  $(\hTa\,v)_\s(x^+_{i}(z))=v_{\tau \s}(x^+_{i+1}(-z))$ for $1\leq i\leq N-2$. 
		        
		By definition, we have
		\begin{align*}
			  & (\hTa\,v)_\s(x^+_{i}(z))=\hTa_{\s}(E_i(d^{\mu_\s(i)}z))=E_{i+1}(-d^{\mu_\s(i)-s_{N}}z), \\
			  & v_{\tau \s}(x^+_{i+1}(-z))=E_{i+1}(-d^{\mu_{\tau\s}(i+1)}z).                             
		\end{align*}
		But $\tau\s=(s_N,s_1,\dots,s_{N-1})$, thus
		\begin{align*}
			\mu_{\tau\s}(i+1)=-s_N-s_1-\cdots -s_i= \mu_{\s}(i)-s_N. 
		\end{align*}
		        
		The proofs for $x^-_{i}(z)$ and $k^\pm_i(z)$ are analogous.
		        
		Equation \eqref{tau Tv} for $i\neq j$ follows from Lemma \ref{T cur} and equation \eqref{tauv}.
		        
		To show \eqref{tau Tv} with $i=j$, set  $l=i-1$ if $i\neq 1$, and $l=2$ if $i=1$. In particular, $A^\s_{l,i}\neq 0$.
		Then, since $l\neq i$, we have
		$$(\hTa\,v\,T_{i})_\s\, (h_{l,\pm 1})=-(v\,T_{i+1})_{\tau \s}\, (h_{l+1,\pm1}).$$
		        
		Also, by a direct computation, we have 
		$$(\hTa\,v\,T_{i})_\s\, (x^\pm_{i,0})=(v\,T_{i+1})_{\tau \s}\, (x^\pm_{i+1,0}).$$
		Therefore, the constant terms of left hand side and right hand side of \eqref{tau Tv} coincide. The equality of other terms follow from the commutator
		$$[h_{l,\pm1},x^+_i(z)]={[A^{\s}_{l,i}]}c^{-(1\pm 1)/2}z^{\pm 1}x^\pm_i(z)$$
		and a similar one for $x^-_{i}(z)$.
	\end{proof}
\end{prop}

\medskip

The homomorphisms $v_\s, h_\s$ and $\hTa_\s$ previously defined correspond to the algebra $\Es(q_1,q_2,q_3)$.
Let $v'_\s, h'_\s$ and $\hTa'_\s$ be the analogous homomorphisms corresponding to the algebra $\Es(q_3,q_2,q_1)$, i.e., the parameter $d$ is switched to $d^{-1}$. 

The map $\eta_\s$ defined in Lemma \ref{etalem} has the following toroidal counterpart.

For $\s\in \Smn$, there exists an anti-isomorphism  $\hat\eta_\s: \Es(q_1,q_2,q_3)\rightarrow \Es(q_3,q_2,q_1)$ given by
\begin{align*}
	  & \hat\eta_\s(C)= C, &   & \hat\eta_\s(K_i^{\pm}(z))= K_i^{\mp}(Cz^{-1}), &   & \hat\eta_\s(E_i(z))= E_i(z^{-1}), &   & \hat\eta_\s(F_i(z))=F_i(z^{-1}) &   & (i\in \hat{I}). 
\end{align*}
The anti-isomorphism $ \hat\eta_\s$ preserves $\deg_i$, $i\in \hat{I}$, and $\deg_\delta( \hat\eta_\s(X))=-\deg_\delta( X)$ if $X\in \Es$ is homogeneous.

\begin{lem}{\label{heta}}
	The anti-isomorphisms $\hat\eta_\s$, $\s\in \Smn$, satisfy
	\begin{align}{\label{heta eq}}
		  & (\hat\eta\, v)_\s=(v'\, \eta)_\s, &   & (\hat\eta\, h)_\s=(h'\, \varphi)_\s, &   & (\hat\eta\, \hTa)_{\s}=(\hTa'\, \hat\eta)_\s. 
	\end{align}
	\begin{proof}
		The equality $(\hat\eta\, h)_\s=(h'\, \varphi)_\s$ is clear.
		
		We check, for example, $(\hat\eta\, v)_\s=(v'\, \eta)_\s$ on $x^+_i(z)$ and $(\hat\eta\, \hTa)_{\s}=(\hTa'\, \hat\eta)_\s$ on $E_{i,r}$. The other cases are analogous. Note that $\hat\eta_\s$ interchanges the parameters $q_1$ and $q_3$, i.e., $d$ and $d^{-1}$ are interchanged. Thus,
		\begin{align*}
			  & (\hat\eta\, v)_\s(x^+_{i,r})= \hat\eta_\s(d^{-r\mu_\s (i)} E_{i,r} )=d^{-r\mu_\s (i)} E_{i,-r}=v'_\s(x^+_{i,-r})=(v'\, \eta)_\s(x^+_{i,r}),                  \\
			  & (\hat\eta\, \hTa)_{\s}(E_{i,r})=\hat\eta_{\tau \s}(-d^{r\,s_N}E_{i+1,r})=-d^{r\,s_N}E_{i+1,-r}=\hTa'_{\s}(E_{i,-r})=(\hTa'\, \hat\eta)_\s(E_{i,r}). 
		\end{align*}
		
	\end{proof}
\end{lem}

\medskip

The maps $\cX_{i,\s}$ defined in Proposition \ref{X action} also have toroidal analogs as follows.

There exist automorphisms of superalgebras $\hX_{i,\s}:\Es \rightarrow \Es$, $i\in \hat{I}$, $\s\in \Smn$, given by 
\begin{align}{\label{X hat}}
	  & \hX_{i,\s}(C)=C, &   & \hX_{i,\s}(X^\pm_j(z))=((-1)^i z^{\mp 1})^{\delta_{ij}}X^\pm_j(z), &   & \hX_{i,\s}(K^\pm_{j}(z))=C^{\mp \delta_{ij}}K^\pm_{j}(z) &   & (j\in \hat{I}). 
\end{align}
The automorphism $ \hX_{i,\s}$ preserves $\deg_j$, $j\in \hat{I}$, and $\deg_\delta( \hX_{i,\s}(X))=\deg_\delta( X)-\deg_i(X)$ if $X\in \Es$ is homogeneous. Let also $\hX_{i,\s}'$ be the analogous automorphism corresponding to the algebra $\Es(q_3,q_2,q_1)$.

Let  $\zeta_\s:\Es \rightarrow \Es$, $\s\in \Smn$, be the rescaling automorphism  given by 
\begin{align*}
	  & \zeta_\s(C)=C, &   & \zeta_\s(X^\pm_i(z))=((-1)^N d^{n-m})^{\pm \delta_{i,0}}X^\pm_i(z), &   & \zeta_\s(K^\pm_i(z))=K^\pm_i(z) &   & (i\in \hat{I}). 
\end{align*}

\begin{prop}{\label{hX}} 
	The automorphisms $\hX_{i,\s},\, \zeta_\s$ satisfy
	\begin{align}
		  & (\hat\eta'\, \hX_{i}'\, \hat\eta)_\s=\hX_{i,\s}^{-1}\,, &   & (\hX_{j}\,v)_\s=(v\,\cX_{j})_\s                  &   &   & (i\in \hat{I},\, j\in I), \label{Xv} \\
		  & (\hTa\, \hX_{j-1})_\s=(\hX_{j}\, \hTa)_\s \,,         &   & (\hTa\, \hX_{N-1})_\s=(\zeta\,\hX_{0}\, \hTa)_\s &   &   & (j\in I). \label{Xtau}               
	\end{align}
	\begin{proof}
		Identities \eqref{Xv} and the first equality of \eqref{Xtau} are clear.
		
		We check $(\hTa\, \hX_{N-1})_\s=(\zeta\,\hX_{0}\, \hTa)_\s$ applied to $X_{N-1,r}^\pm$:
		\begin{align*}
			  & (\hTa\, \hX_{N-1})_\s(X_{N-1,r}^\pm)=\hTa_\s((-1)^{N-1} X^\pm_{N-1,r\mp 1})=(-1)^{N+r}d^{(r\pm1)(m-n)} X^\pm_{0,r\mp 1},                                                            \\
			  & (\zeta \hX_{0}\, \hTa)_\s(X^\pm_{N-1,r})=(\zeta\hX_{0})_\s ((-1)^r d^{r(m-n)}X^\pm_{0,r})=\zeta_\s ((-1)^r d^{r(m-n)}X^\pm_{0,r\mp 1})=(-1)^{N+r}d^{(r\pm1)(m-n)} X^\pm_{0,r\mp 1}. 
		\end{align*}
		The check on the remaining generators is similar. 
	\end{proof}
\end{prop}

\section{Toroidal braid group}{\label{tor braid sec}}
We construct an action of the toroidal braid group $\widehat\B_N$
associated with $\mathfrak{sl}_{m|n}$ on $\E_\bullet=\bigoplus_{\s\in \Smn} \Es$.
As a consequence, we show that the algebras $\Es$, $\s \in \Smn$, are all isomorphic. It also gives us the {\it Miki automorphism} of $\E_\s$, which interchanges the horizontal and vertical subalgebras.

In this section, we assume $N\geq 4$.

\subsection{Action of \texorpdfstring{$\B_N$}{BN} on \texorpdfstring{$\E_\bullet$}{E.}} 
We start with extending the action of affine braid group $\B_N$ from the vertical subalgebra $U_q^{ver}\slh_\bullet$ given in Proposition \ref{ybraid} to the toroidal algebra $\E_\bullet$.

Note that the map $\hTa_\s$ was already defined in \eqref{htau}.
We also recall that the prime indicates the action of the operator with $q_3$ and $q_1$ switched, see Section \ref{sec mor}. The following theorem is a supersymmetric analog of \cite[Proposition 1]{M}.

\begin{thm}{\label{Tbraid}} Let $N>3$. For $i\in \hat{I}$, $\s\in \S_{m|n}$, there exists an isomorphism of superalgebras $\hT_{i,\s}:\Es\rightarrow \E_{\sigma_{i} \s}$ satisfying 
	\begin{align}
		  & (\hT_{i\,}v)_{\s}=(v T_{i})_\s                       & (i\in I),\label{P1}          \\
		  & (\hT_{i}\,h)_{\s}=(h\, T_{i})_\s                     & (i\in \hat{I}),\label{P2}    \\
		  & (\hTa\, \hT_{i})_\s=(\hT_{i+1}\hTa)_{\s}             & (i\in \hat{I}),\label{P3}    \\
		  & (\hat\eta'\, \hT_{i}'\, \hat\eta)_\s=(\hT_{i}^{-1})_\s & (i\in \hat{I}).\label{hetaT} 
	\end{align}
	
	Moreover, the isomorphisms $\hT_{i,\s}$ satisfy the Coxeter relations 
	\begin{align}
		  & (\hT_{i+ 1}\hT_{i}\hT_{i+ 1})_\s=(\hT_{i}\hT_{i+ 1}\hT_{i})_\s & (i\in \hat{I}),\label{Cox1} \\
		  & (\hT_{i}\hT_{j})_\s=(\hT_{j}\hT_{i})_\s                        & (i\neq j\pm 1).\label{Cox2} 
	\end{align}
	Finally,  $\hT_{i,\s}$ are graded with respect to homogeneous grading.
	\begin{proof}
		Throughout this proof we write similar formulas for $X_i^\pm (z)$ and $\tilde K_i^\pm(z)$. To avoid repeating the same formula four times, we use the letters  $A_i(z)$  to denote $X^\pm_i(z)$  or $\tilde{K}^\pm_i(z)$, and $a_i(z)$ to denote $x^\pm_i(z)$ or $k^\pm_i(z)$ . Note that $A_i(z)$ and $a_i(z)$ in the same formula are all of the same kind - e.g. all $X_i^+(z)$ and $x_i^+(z)$.
		
		Define the map $\hT_{1,\s}$ on generators of $\Es$ by
		$$\hT_{1,\s}(A_0(z))=(\hTa^{-1}v T_{2})_{\tau \s}\, (a_1(-z)),\qquad \hT_{1,\s}(A_i(z))=(vT_{1})_\s\, (a_i(z)) \qquad(i\in I).$$ 
		Note that $\hT_{1,\s}(A_i(z))=A_i(z)$ if $i=3,\dots,N-1$. Moreover, the action of $\hT_{1,\s}$ on $A_0(z), A_2(z), A_{1,0}$ is explicit by Lemma \ref{T cur}. The map $\hT_{1,\s}$ respects homogeneous grading because $T_{1,\s}$ and $T_{2,\s}$ do.
		
		We claim that this extends to an isomorphism of superalgebras. In fact all relations which do not involve node $1$ (that is the ones which do not contain $E_1(z), F_1(z),K^\pm_1(z)$) can be checked by a direct computation. To check the relations which do involve node $1$ and to reduce the calculations in other cases we can use the following arguments.
		
		The relations not involving node $0$ are satisfied since we can compute in the vertical algebra and $T_{1,\s}$ is a homomorphism. 
		
		To check the relations involving node 0, note that by Proposition \ref{htauprop} we have
		\begin{align*}
			  & \hTa_\s (A_{N-1}(z))=A_{0}(-d^{n-m}z),\qquad \hTa_\s (A_{i}(z))=A_{i+1}(-z)                                          & (i\neq N-1),      \\
			  & \hT_{1,\s}(A_i(z))=(\hTa^{1-k} v T_{k})_{\tau^{k-1}\s}(a_{k+i-1}((-1)^{1-k}z))                                       & (k, k+i-1 \in I), \\
			  & \hT_{1,\s}(A_{N-1}(d^{m-n}z))=(\hTa^{-2} v T_{3})_{\tau^{2}\s}(a_{1}(z))=(\hTa^{-3} v T_{4})_{\tau^{3}\s}(a_{2}(-z)) & (N>4),            \\
			  & \hT_{1,\s}(A_{N-2}(d^{m-n}z))=(\hTa^{-3} v T_{4})_{\tau^{3}\s}(a_{2}(-z))                                            & (N>4).            
		\end{align*}

		Let us first assume that $N>4$.
		
		We prove the relations involving the nodes $0,1$, and $2$ by moving these nodes to $1,2$ and $3$ using $\tau$ and $\hTa$. Namely, by \eqref{eab3} and \eqref{tauv}, we have
		\begin{align*}
			  & \hT_{1,\s}(A_i(z))=(\hTa^{-1} v T_{2})_{\tau\s}(a_{i+1}(-z)) &   & (i=0,1,2). 
		\end{align*}
		Observe that the defining relations between generators of $\E_\s$ involving nodes $0,1,2$ are the same  as the defining relations in $\U_q\slh_{\tau\s}$ involving nodes $1,2,3$. Since $(\hTa^{-1} v T_{2})_{\tau\s}$ is a homomorphism, it maps the the relations of $\U_q\slh_{\tau\s}$ involving the nodes $1,2,3$  to zero. Therefore $\hT_{1,\s}$ maps defining relations of $\E_\s$ involving nodes $0,1,2$ to zero.
		
		\medskip 
		
		The relations involving the nodes $N-1,0,1$  are treated similarly: we go to nodes $1,2,3$ again by using
		\begin{align*}
			  & \hT_{1,\s}(A_i(z))=(\hTa^{-2} v T_{3})_{\tau^{2}\s}(a_{i+2}(z)) &   & (i=0,1), \\
			&\hT_{1,\s}(A_{N-1}(d^{m-n}z))=(\hTa^{-2} v T_{3})_{\tau^{2}\s}(a_{1}(z)).
		\end{align*}
		Note that the defining relations between generators of $\E_\s$ involving nodes $N-1,0,1$ are the same  as the defining relations in $\U_q\slh_{\,\tau^2\s}$ involving nodes $1,2,3$ if the shift of spectral parameter in the generating series related to node $N-1$ is taken into account. Therefore as before, $\hT_{1,\s}$ maps defining relations of $\Es$ involving nodes $N-1,0,1$ to zero.
		
		For the relations involving the nodes $N-2,N-1,0$ we proceed in the same way by using
		\begin{align*}
			  & \hT_{1,\s}(A_{N-i}(d^{m-n}z))=(\hTa^{-3} v T_{4})_{\tau^{3}\s}(a_{3-i}(-z)) &   & (i=1,2), \\
			&\hT_{1,\s}(A_{0}(z))=(\hTa^{-3} v T_{4})_{\tau^{3}\s}(a_{3}(-z)).
		\end{align*}
		and reducing to nodes $1,2,3$ once again. We omit further details. Thus for $N>4$ all relations follow without extra computations.
		
		If $N=4$, the previous argument applies for the relations involving the nodes $0,1,2$, or $-1,0,1$, or $0,1$ or the nodes involving $2,0$.  The additional relations $\eqref{Serre3}$ and $\eqref{Serre4}$ for $i=3$ in the case $A^\s_{3,3}=0$ are checked directly. We check $\eqref{Serre3}$ with $i=3$ as an example.
		
		First, by identity \eqref{q der} and relation $\eqref{Serre3}$ we have 
		\begin{align*}
			0 & =\Sym_{{z_1,z_2}}\lb{\lb{\lb{E_3(z_1),\lb{E_{0}(w_1),\lb{E_3(z_2),E_{2}(w_2)}}},E_{1,0}},E_{1,0}}            \\
			  & =(q+q^{-1})\Sym_{{z_1,z_2}}\lb{\lb{E_3(z_1),\lb{E_{0}(w_1),E_{1,0}}},\lb{E_3(z_2),\lb{E_{2}(w_2),E_{1,0}}}}. 
		\end{align*}
		Using Lemma \ref{T cur} we can compute the action of $\hT_{1,\s}$ on $\eqref{Serre3}$ explicitly. Note that $A^\s_{3,3}=0$, $N=4$, $m\neq n$ imply $mn=3$ and $|1|=0$. We have
		\begin{align*}
			  & \hT_{1,\s}(\Sym_{{z_1,z_2}}\lb{E_3(z_1),\lb{E_{0}(w_1),\lb{E_3(z_2),E_{2}(w_2)}}})                                  \\
			  & =\hT_{1,\s}(\Sym_{{z_1,z_2}}\lb{\lb{E_3(z_1),E_{0}(w_1)},\lb{E_3(z_2),E_{2}(w_2)}})                                 \\
			  & =q^{-A^\s_{1,1}}\Sym_{{z_1,z_2}}\lb{\lb{E_3(z_1),\lb{E_{0}(w_1),E_{1,0}}},\lb{E_3(z_2),\lb{E_{2}(w_2),E_{1,0}}}}=0. 
		\end{align*}
		
		This shows that $\hT_{1,\s}$ is a homomorphism. 
		
		\medskip
		
		For $i\in \hat{I}$ define
		\begin{align}{\label{Ti def}}
			\hT_{i,\s}=(\hTa^{\,i-1}\, \hT_{1}\, \hTa^{\,1-i})_\s. 
		\end{align}
		Since $\hT_{1,\s}$ and $\hTa_\s$ are homomorphisms for all $\s\in \Smn$, the maps $\hT_{i,\s}$ are well defined homomorphisms for all $i\in \hat{I}$ and $\s \in \Smn$.
		
		Note that  
		\begin{align}{\label{fixT}}
			  & \hT_{i,\s}(A_j(z))=A_j(z) & (j\neq i,i\pm 1). 
		\end{align} Also, $\hT_{i,\s}(A_j(z))$ is explicit if $j= i\pm1 $ and so is $\hT_{i,\s}(A_{i,0})$.
		
		Now, we show that the homomorphisms $\hT_{i,\s}$ satisfy equations \eqref{P1} using induction on $i$.  For $i=1$ the statement follows from definition of $\hT_1$. Suppose \eqref{P1} is true for $i=j\leq N-2$. Let us prove it for $i=j+1$.
		
		If $l \in I$, $l\neq 1$, we have
		\begin{align*}
			(\hT_{j+1} v)_\s(a_{l}(z)) & =(\hTa\, \hT_{j}\, \hTa^{-1} v)_{\s} (a_l(z))=(\hTa\, \hT_{j} v)_{\tau^{-1}\s} (a_{l-1}(-z)) \\
			                           & =(\hTa\, v\, T_j)_{\tau^{-1}\s} (a_{l-1}(-z))=(v\, T_{j+1})_{ \s} (a_{l}(z)).                
		\end{align*}
		Here the second equality is \eqref{tauv}, the third equality is the induction hypothesis, and the last equality is \eqref{tau Tv}.
		
		If $l=1$ then
		\begin{align*}
			  & (\hT_{j+1}\, v)_\s(a_1(z))=(\hTa \, \hT_{j}\, \hTa^{-1})_{\s} (A_1(z))=(\hTa\, \hT_j)_{\tau^{-1}\s} (A_{0}(-z))=(v T_{j+1})_{ \s} (a_{1}(z)). 
		\end{align*}
		Here the last equation is a definition if $j=1$ and a trivial statement if $j>1$ since in that case $\hT_{j,\s} (A_0(z))=A_0(z)$ and $T_{j+1,\s}(a_1(z))=a_1(z).$
		
		Thus, $\hT_{i,\s}$ satisfy equation \eqref{P1}. 
		
		\medskip
		
		Next we show \eqref{hetaT} for $i=1$.
		By \eqref{P1}, Lemmas \ref{etalem} and \ref{heta}, we have
		\begin{align*}
			  & (\hT_{1}\hat\eta' \hT_{1}' \hat\eta)_\s(A_i(z))=(v\, T_1 \,\eta\, T_{1}\, \eta)_\s (a_i(z))= v_\s(a_i(z))=A_i(z) &   & (i\in I), \\
			&(\hT_{1}\hat\eta' \hT_{1}' \hat\eta)_\s(A_0(z))=(\hTa ^{-1}v\, T_2 \,\eta\, T_{2}\, \eta)_{\tau\s} (a_1(-z))= (\hTa^{-1}\,v)_{\tau\s}(a_1(-z))=A_0(z).
		\end{align*}
		Thus, equation \eqref{hetaT} holds for $i=1$. In particular, it implies that $\hT_{1,\s}$ has an inverse and therefore is an isomorphism. 
		
		By Lemma \ref{heta}, $\hat\eta_\s$ commutes with $\hTa_\s$. It implies \eqref{hetaT} for all $i\in \hat I$. In particular, $\hT_{i,\s}$ is isomorphism for all $i\in\hat I$.
		
		\medskip
		
		By \eqref{tauh}, the isomorphisms $\hT_{i,\s}$ satisfy equation \eqref{P2}.
		
		\medskip
		
		Finally, we show Coxeter relations \eqref{Cox1} and \eqref{Cox2}.
		By equation \eqref{P3}, it is sufficient to show these relations when $i=1$ and $j\neq 0$. By Proposition \ref{braidaff}, Coxeter relations are satisfied by the homomorphisms $T_{i,\s}$. Then by \eqref{P1}, relations
		\eqref{Cox1} and \eqref{Cox2} are satisfied on the image of $v_\s$. By \eqref{P2}, these relations are also satisfied on the image of $h_\s$. Since horizontal and vertical subalgebras generate the whole algebra $\E_\s$, we obtain the proof  of \eqref{Cox1} and \eqref{Cox2} 
		
	\end{proof}
\end{thm}

\begin{cor}{\label{iso cor}}
	The superalgebras $\Es$ are isomorphic for all $\s \in \Smn$.
\end{cor}
\begin{proof}
	The corollary follows from Theorem \ref{Tbraid}. 
\end{proof}

\begin{rem}The corollary above treats the case $N\geq 4$. For $N=3$, the isomorphisms between all three algebras  $\Es$ are given by the map $\hTa$.
\end{rem}

\begin{cor} Let $N>3$.
	The automorphisms  $\hTa$, $\hT_i$, $i\in \hat{I}$, define an action of the extended affine braid group $\B_N$ on $\E_\bullet$, i.e., they satisfy the relations \eqref{eab1}-\eqref{eab3}.
\end{cor}

\subsection{Toroidal braid group}	

We recall the definition of the toroidal braid group of $\sl_{N}$, cf. \cite{M}.

\begin{dfn}
	The toroidal braid group $\widehat\B_N$ of $\sl_{N}$ is the group generated by elements $\hTa$, $\hT_{i}$, $\cY_{j}$, $i\in I, j\in \hat I$, satisfying the relations
	\begin{align}
		  & \hT_i\hT_{j}=\hT_j\hT_{i}           & (j\neq i,i\pm 1),\label{tb1}    \\
		  & \hT_j\hT_i\hT_{j}=\hT_i\hT_j\hT_{i} & (j=i\pm 1), \label{tb2}         \\
		&\cY_i\cY_j=\cY_j\cY_i, \label{tb3} \\
		  & \hT_i\cY_j=\cY_j\hT_i               & (j\neq i, i+1), \label{tb4}     \\
		  & \hT_i^{-1}\cY_i\hT_i^{-1}=\cY_{i+1} & (i\in I), \label{tb5}           \\
		  & \hTa \hT_i \hTa^{-1}=\hT_{i+1},     & (1\leq i \leq N-2), \label{tb6} \\
		&\hTa^2 \hT_{N-1} \hTa^{-2}=\hT_{1}, \label{tb7} \\
		  & \hTa \cY_i \hTa^{-1}=\cY_{i+1}      & (i\in I).\label{tb8}            
	\end{align}
\end{dfn}

We remark that the quotient of the toroidal braid group $\widehat\B_N$ by the relation $\hTa \cY_0 \hTa^{-1}=\cY_1$ is isomorphic to double affine Hecke group with central element set to 1, see Definition 4.1 in \cite{C}.

The toroidal braid group has the following Fourier transform given by qKZ elements.
\begin{lem}{\cite{C}\cite[(2.38)]{M}\label{tor braid auto}}
	There exists an automorphism $\Phi$ of $\widehat\B_N$ given by
	\begin{align*}
		  & \Phi(\hT_i)=\hT_i, &   & \Phi(\cY_j)=\hT_{j-1}^{-1}\cdots \hT_{1}^{-1}\, \hTa \, \hT_{N-1}\cdots \hT_{j}, &   & \Phi(\hTa)=\cY_1^{-1}\hT_{1}\cdots \hT_{N-1}. 
	\end{align*} \qed
\end{lem}

Note that the subgroup $G\subset \widehat\B_N$ generated by $\hT_1$ and $\hTa$, and the subgroup $H\subset \widehat\B_N$ generated by $\hT_i$, $i\in I$, and $\cY_1$ are both isomorphic to the extended affine braid group $\B_N$. The isomorphism $\gamma$ between these two presentations of $\B_N$ is described in \eqref{braideq}.

Let $i_G$ be the inclusion $i_G: G\cong \B_N\rightarrow \widehat\B_N$ given by
\begin{align*}
    i_G(T_1)=\hT_1, \qquad i_G(\tau)=\hTa,
\end{align*}
and  $i_H$ be the inclusion $i_H: H\cong \B_N\rightarrow \widehat\B_N$ given by
\begin{align*}
    i_H(\cX_1)=\cY_1, \qquad i_H(T_i)=\hT_i\quad & (i\in I).
\end{align*}

The following lemma is easily checked on the generators.

\begin{lem}{\label{phi diag}}
    The homomorphisms $\gamma$, $i_G$, $i_H$ and $\Phi$ satisfy the following commutative diagram
   \begin{center}
       
    \begin{tikzcd}
\widehat{\B}_N \arrow[r,"\sim", "\Phi"' ]
& \widehat{\B}_N  \\
\B_N \arrow[u,hook, "i_H" ]  \arrow[r, "\sim", "\gamma"']
& \B_N \arrow[u,hook, "i_G"']
\end{tikzcd}
\end{center}
\qed
\end{lem}

Recall automorphisms $\zeta_\s$, $\hX_{i,\s}$ of $\Es$ described in Proposition \ref{hX}. Define the following automorphisms of $\Es$
\begin{align}{\label{Y hat}}
	  &   & \cY_{0,\s}=(\zeta \hX_0 \hX_{N-1}^{-1})_\s\, , &   & \cY_{i,\s}=(\hX_i \hX_{i-1}^{-1})_\s &   & (i\in I). 
\end{align}

The following is the super-analog of \cite[Corollary 1]{M}.
\begin{prop}
	The automorphisms $\hTa, \hT_i, \cY_j$, $i\in I$, $j\in \hat{I}$, define an action of the toroidal braid group $\widehat\B_N$ on $\E_\bullet$.
	\begin{proof}
		The relations for $\hT_i$ and $\hTa$ follow from Theorem \ref{Tbraid}. Relation \eqref{tb8} between $\cY_i$ and $\hTa$ follows from \eqref{Xtau}. Relations \eqref{tb4} are clear due to \eqref{fixT}. Equation \eqref{tb5}  between $\cY_i$ and $\hT_j$ on vertical subalgebra follows from \eqref{eab10} and \eqref{eab8} (note that $(\hX_{0}\,v)_\s=v_\s$). To check the relations on $A_0(z)$ we write $A_0(z)=\hTa^{-1}_\s (A_1(-z))$ and use the already established relations with $\hTa$.
	\end{proof}
\end{prop}

\subsection{Miki automorphism}

Now we are ready to prove the existence of the Miki automorphism of $\E_\s$ which switches horizontal and vertical subalgebras.

Recall the isomorphism $\iota_\s$ identifying the new Drinfeld and Drinfeld-Jimbo realizations of $U_q\slhs$, see Proposition \ref{X}.

The following is the super-analog of \cite[Theorem 1]{M}.

\begin{thm}{\label{mikithm}} Let $N>3$.
	There exists a superalgebra automorphism $\psi_\s$ of $\Es$ satisfying
	\begin{align}{\label{psi v h}}
		(\psi\, v)_\s= (h\, \iota)_\s, &   & (\psi\, h\, \iota)_\s=(v\, \eta\, \iota^{-1}\, \varphi\, \iota)_\s, &   & \psi^{-1}_\s= (\hat{\eta}'\, \psi'\, \hat{\eta})_\s, 
	\end{align}
	where the first two equalities are equalities of maps from the new Drinfeld realization of $U_q\slhs$ to $\Es$.
	\begin{proof}
		We often write equation of maps from $U_q\slhs$ to $\Es$, similar to \eqref{psi v h}. We understand that new Drinfeld realization is identified with the Drinfeld-Jimbo realization via map $\iota$ and do not distinguish between them. In particular, we skip $\iota$ from our formulas.
		
		Recall notation \eqref{X not}.
		For $i\in I$, let  $\cZ_{i,\s}=(\cY_1\cdots \cY_{i})_\s$.
		
		Using equations \eqref{braideq}, \eqref{tauh} and \eqref{P2} we get 
		\begin{align*}
			\Phi(\cY_{1,\s})h_\s=h_{\s}(\tau T_{N-1}\cdots T_{1})_\s=(h\,\cX_{1})_\s\,. 
		\end{align*}
		And, by relations \eqref{eab8}-\eqref{eab10}, we have
		\begin{align*}
			\Phi(\cY_{i,\s})h_\s=(h\, \cX_{i}\, \cX_{i-1}^{-1})_\s &   & (i\in I\setminus\{1\}). 
		\end{align*}
		Thus, 
		\begin{align}
			  & \Phi(\cZ_{i,\s})h_\s=h_{\s}\cX_{i,\s} &   & (i\in I) \label{Z1}. 
		\end{align}
		
		Define  $\psi_\s$ on the $\Es$ generators by
		\begin{align*}
			  & \psi_\s(X^\pm_{i,r})=(-1)^{ir}\Phi(\cZ_i^{\mp r})_\s (X^\pm_{i,0}),              &   & \psi_\s(K_{i})=K_{i}                          &   & (i\in I, r\in \Z), \\
			  & \psi_\s(X^\pm_{0,r})=(-1)^r \Phi(\hTa^{-1}\cZ_1^{\mp r})_{\tau \s}(X^\pm_{1,0}), &   & \psi_\s(K_{0})=\Phi(\hTa^{-1})_{\tau \s}(K_{1}) &   & (r\in \Z).         
		\end{align*}
		
		For $i\in I$, $r\in \Z'$, by equation \eqref{Z1}, we have
		\begin{align*}
			  & (\psi\, v)_\s (x^{\pm}_{i,r})=(-1)^{ir}\Phi(\cZ_i^{\mp r})_\s (X^\pm_{i,0})=(-1)^{ir}(\Phi(\cZ_i^{\mp r})h)_\s (x^\pm_{i,0})=(-1)^{ir}h_{\s}\cX_{i,\s}^{\mp r}(x^\pm_{i,0})=h_\s (x^{\pm}_{i,r}). 
		\end{align*}
		
		This implies  $\psi_\s v_\s= h_\s$. Thus, $\psi_\s$ extends to a homomorphism  $U^{ver}_q\slhs\to \Es$.
		
		We now check that $\psi_\s$ satisfy the relations involving the node $0$. This is done in a similar way as in Theorem \ref{Tbraid}. For $i\in I$, using \eqref{TTX=X} and the relations in the group, we obtain
		\begin{align*}
			  & \Phi(\hTa)_\s (X^\pm_{i,0})=(\cY_1^{-1}\hT_{1}\cdots \hT_{N-1})_\s (X^\pm_{i,0})=X^\pm_{i+1,0} & (i\neq N-1), \\
			&\Phi(\hTa^2)_\s (X^\pm_{N-1,0})=(\cZ_2^{-1}\hT_{2}\hT_{1}\hT_{3}\hT_{2}\cdots \hT_{N-1}\hT_{N-2})_\s (X^\pm_{N-1,0})=X^\pm_{1,0}.
		\end{align*}
		Thus, the relations involving the nodes $0,1$ and $2$ follow from the relations involving the nodes $1,2$ and $3$ in $U^{ver}_q\slh_{\,\tau \s}$ using the equation
		\begin{align}{\label{psi tau}}
			  & \psi_\s(X^\pm_{i,r})=(-1)^{r} \Phi(\hTa^{-1})_{\tau \s}\psi_{\tau \s}(X^\pm_{i+1,r}) & (i\in \hat{I}). 
		\end{align}
		
		For the relations involving the nodes $0,1$ and $N-1$  we use 
		\begin{align*}
			  & \psi_\s(X^\pm_{i,r})=(-1)^r(d^{m-n})^{\pm r\delta_{i,N-1}} \Phi(\hTa^{-2})_{\tau^2 \s}\psi_{\tau^2 \s}(X^\pm_{i+2,r}) &   & (i=0,1, N-1). 
		\end{align*}
		
		And for the relations involving the nodes $0,N-1$ and $N-2$  we use 
		\begin{align*}
			  & \psi_\s(X^\pm_{i,r})=(-1)^r(d^{m-n})^{\pm r(1-\delta_{i,0})} \Phi(\hTa^{-3})_{\tau^3 \s}\psi_{\tau^3 \s}(X^\pm_{i+3,r}) &   & (i=0,N-1, N-2). 
		\end{align*}
		
		\medskip
		
		We check the equation $ (\psi\, h)_\s=(v\, \eta\, \varphi)_\s$ on the Chevalley generators $e_i$, $i\in \hat{I}$. The proof for $f_i, t_i$, $i\in \hat{I}$, is analogous. By \eqref{braideq} and \eqref{DJnDe}, we have
		\begin{align*}
			&\begin{aligned}
			  & (\psi\, h)_\s(e_i)=\psi_\s(X^+_{i,0})=X^+_{i,0}=v_\s (x^+_{i,0})=(v\, \eta\, \varphi)_\s(e_i) &   &   & (i\in I), 
			\end{aligned}\\
			&\begin{aligned}
			&(\psi\, h)_\s(e_0)=\psi_\s(X^+_{0,0})=\Phi(\hTa^{-1})_{\tau \s}(X^\pm_{1,0})=(\hT_{N-1}^{-1}\cdots \hT_{1}^{-1}\hX_1\hX_2^{-1})_{\tau\s}(X^+_{1,0})=(\hT_{N-1}^{-1}\cdots \hT_{1}^{-1}\hX_1)_{\tau\s}(X^+_{1,0}),\\
			&(v\, \eta\, \varphi)_\s(e_0)=v_\s \eta_\s(T_{N-1}\cdots T_{1}\cX_1^{-1})_\s(x^+_{1,0})=v_\s (T_{N-1}^{-1}\cdots T_{1}^{-1}\cX_1)_{\tau\s}(x^+_{1,0})=(\hT_{N-1}^{-1}\cdots \hT_{1}^{-1}\hX_1)_{\tau\s}(X^+_{1,0}).
			\end{aligned}
		\end{align*}
		
		\medskip
		
		Finally, we show $\psi^{-1}_\s= (\hat{\eta}'\, \psi'\, \hat{\eta})_\s.$
		
		By Lemma \ref{heta} and the identities $ (\psi\, v)_\s= h_\s,$ $(\psi\, h)_\s=(v\, \eta\, \varphi)_\s$, we have
		\begin{align*}
			  & (\psi\, \hat{\eta}'\, \psi' (\hat{\eta}\, v))_\s=(\psi\, \hat{\eta}' (\psi'\, v') \eta)_\s=(\psi (\hat{\eta}'\, h') \eta)_\s=((\psi\, h) \varphi\, \eta)_\s=(v\, \eta\, \varphi\,\varphi\, \eta)_\s=v_\s\,, \\
			  & (\psi\, \hat{\eta}'\, \psi'(\hat{\eta} h))_\s=(\psi\, \hat{\eta}' (\psi'\, h') \varphi)_\s=(\psi (\hat{\eta}'\,v')\eta)_\s=(\psi\, v)_\s=h_\s\,.                                                            
		\end{align*}
		
		Thus, $(\psi\, \hat{\eta}'\, \psi'\, \hat{\eta})_\s=1_\s$ on both $U^{ver}_q\slh_{\s}$ and $U^{hor}_q\slh_{\s}$, but they generate $\Es$. Therefore $(\psi\, \hat{\eta}'\, \psi'\, \hat{\eta})_\s=1_\s$ on $\Es$.
		
		This completes the proof.
	\end{proof}    
\end{thm}    

Since the Miki automorphism $\psi_\s$ sends the vertical subalgebra $\U_q^{ver}\slhs$ to the horizontal $\U_q^{hor}\slhs$  subalgebra, from Proposition \ref{lemv} we obtain the following corollary.

\begin{cor}{\label{h inj}} Let $N>3$.
	For generic values of parameters, the horizontal map $h_\s:\U_q\slhs \rightarrow \Es$ is injective. In particular, $\U_q^{hor}\slhs$ is isomorphic to $\U_q\slhs$.\qed
\end{cor}
The key property used in the proof of Theorem \ref{mikithm}  was the compatibility of $\psi$ with the braid group action which we describe in the next proposition. 
\begin{prop}{\label{miki phi}}
	The automorphism $\psi_\s$ satisfies
	\begin{align}{\label{psi phi}}
		  & (\psi B)_\s=(\Phi(B)\psi)_\s & (B\in \widehat \B_N). 
	\end{align}
	
	\begin{proof}
		If $B=\hT_i$, $i\in I$, we use \eqref{P1}, \eqref{P2} and the first two equalities of \eqref{psi v h} to show \eqref{psi phi} is satisfied on the horizontal and vertical subalgebras. We have
		\begin{align*}
			  & (\psi\, \hT_i\,v)_\s=(\psi\, v\,T_i)_\s=(h\,T_i)_\s=(\hT_i\,h)_\s=(\hT_i\, \psi\,v)_\s=(\Phi(\hT_i)\, \psi\,v)_\s ,                              \\
			  & (\psi\, \hT_i\,h)_\s=(\psi\, h\,T_i)_\s=(v\,\eta\,\varphi\,T_i)_\s=(\hT_i\,v\,\eta\,\varphi)_\s=(\hT_i\, \psi\,h)_\s=(\Phi(\hT_i)\, \psi\,h)_\s. 
		\end{align*}
		Since the horizontal and vertical subalgebras generate $\Es$, we have $(\psi \hT_i)_\s=(\Phi(\hT_i)\psi)_\s$ on $\Es$.
		
		In the case $B=\hTa$, equation \eqref{psi phi} is precisely the equation \eqref{psi tau}.
		
		Since $\widehat{\B}_N$ is generated by $\hT_i, \cY_1$ and $\hTa$, it remains to check \eqref{psi phi} for $B=\cY_1$.
		From the previous cases we have
		\begin{align*}
			(\psi\, \hTa)_\s=(\Phi(\hTa)\,\psi)_\s=(\cY^{-1}_1\, \hT_1\cdots \hT_{N-1}\,\psi)_\s. 
		\end{align*}
		Thus,
		\begin{align*}
			(\hTa^{-1}\,\psi^{-1} )_\s=(\psi^{-1}\, \hT_{N-1}^{-1}\cdots \hT_{1}^{-1}\,\cY_1)_\s=(\hT_{N-1}^{-1}\cdots \hT_{1}^{-1}\,\psi^{-1}\, \cY_1)_\s, 
		\end{align*}
		or equivalently
		\begin{align*}
			(\psi^{-1} \,\cY_1^{-1})_\s=(\hTa\,\hT_{N-1}^{-1}\cdots \hT_{1}^{-1}\,\psi^{-1})_\s. 
		\end{align*}
		
		By the first equality of \eqref{Xv} and the last equality of \eqref{psi v h}, we have
		\begin{align*}
			(\psi \,\cY_1)_\s =(\hat \eta' (\psi')^{-1} (\cY_1')^{-1}\hat \eta )_\s, 
		\end{align*}
		and by equation \eqref{hetaT} and the last equality of \eqref{heta eq}, we have
		\begin{align*}
		(\hat \eta' (\psi')^{-1} (\cY_1')^{-1}\hat \eta )_\s=(\hat \eta'\, \hTa'(\hT_{N-1}')^{-1}\cdots (\hT_{1}')^{-1}(\psi')^{-1}\hat \eta)_\s=(\hTa\, \hT_{N-1}\cdots \hT_{1}\,\psi)_\s=(\Phi(\cY_1)\,\psi)_\s. 
		\end{align*}
	\end{proof}
\end{prop}

Finally, we describe how Miki automorphism changes the grading, cf. \cite[Theorem 1]{M}.
\begin{prop} If $X\in \Es$ is homogeneous, then 
	\begin{align}\label{grad tor}
		\deg_\delta(\psi_\s(X))=-\deg_0(X), \qquad \deg_i(\psi_\s(X))=\deg_\delta(X)+\deg_i(X)-\deg_0(X) \quad  (i\in \hat{I}). 
	\end{align}
\end{prop}
\begin{proof}
Recall \eqref{v deg} and \eqref{h deg}.

If $X\in U_q^{ver}\slhs$, then  \eqref{grad tor} reduces to \eqref{iota deg}. Similarly if $X\in U_q^{hor}\slhs$, then  \eqref{grad tor} reduces to \eqref{iota^-1 deg} thanks to \eqref{eta deg}.

	The proposition follows since vertical and horizontal subalgebras generate $\Es$.
\end{proof}

\appendix
\allowdisplaybreaks
	
\section{}\label{app 1}
In this appendix, we define another action of the braid group of finite type $A$ on a suitable completion $\widetilde{U}_q\slh_\bullet$ of $U_q\slh_\bullet$. We write an evaluation homomorphism from $\Es$ to $\Uge$ with parity $\s$ in terms of the braid group action.

\subsection{The superalgebra \texorpdfstring{$\Ug$}{Uqglm|n}}

Let $\s$ be a parity sequence. We denote the superalgebra $\Ug$ with parity $\s$ by $\U_q\glhs$. Here we use a presentation of $\Ug$ similar to the presentation of $U_q\widehat{\gl}_m$ given in \cite{DF}.

 The algebra $\U_q\glhs$ is obtained by adding to $U_q\slhs$ a Heisenberg current $H(z)$ commuting with it and extending the root lattice generated by $k_{i}$. For our purposes it is convenient to write the generators and relations in the following way.

Define the matrix $B^\s=(B_{i,j}^\s)_{i,j \in \hat I}$, where $B_{i,j}^\s=s_i(\delta_{i,j}-\delta_{i,j+1})$.
Note that $A^\s_{i,j}=B_{i,j}^\s-B_{i+1,j}^\s$. 

For $i,j \in \hat{I}$, define $\delta_{i>j}=1$ if $i>j$ and $\delta_{i>j}=0$ if $i\leq j$. In this notation, we still consider the elements of $\hat{I}$ modulo $N$, but we identify $\hat{I}$ with the set $\{1,2,\dots,N-1,N\}$. For example, $\delta_{0>1}=\delta_{N>1}=1$.

\medskip

The superalgebra $\U_q\glhs$ is generated by current generators $c$, $\phi^{\pm 1}_i$, $\phi_{i,r}$, $x^\pm_{j,r}$, $i\in \hat{I}$, $j\in I$, $r\in \Z'$. 

Let $\phi^\pm_i(z)=\phi_i^{\pm 1}\exp \left(\pm (q-q^{-1})\sum_{r>0}\phi_{i,\pm r}z^{\mp r}\right)$, $i\in \hat{I}$.

For $i\in I$, define $k^\pm_i(z)=\phi^\pm_i(q^{\mu_\s(i)}z)\phi^\pm_{i+1}(q^{\mu_\s(i)}z)^{-1}$, where $\mu_\s(i)=-\sum_{j=1}^i s_j$, see Section \ref{hv sec}.

The defining relations are as follows. First, the currents
$k^{\pm}_i(z)$, $x^\pm_i(z)$, $i\in I$ satisfy the relations of  $U_q\slhs$ given in the new Drinfeld realization as in Section \ref{sl section}.

The remaining defining relations are as follows.

For $i,j\in \hat{I}$
\begin{align*}
	  & \phi^\pm_i(z)\phi^\pm_j(w)=\phi^\pm_j(w)\phi^\pm_i(z),                                                                                                                                                                                       \\
	  & \dfrac{c^{-1}z-w}{cz-w}\dfrac{c^{-1}z-q^{2s_i\delta_{j>i}}w}{cz-q^{2s_i\delta_{j>i}}w}\phi^+_i(z)\phi^-_j(w)=\dfrac{c^{-1}z-q^{2s_i}w}{cz-q^{2s_i}w}\dfrac{c^{-1}z-q^{-2s_i\delta_{i>j}}w}{cz-q^{-2s_i\delta_{i>j}}w}\phi^-_j(w)\phi^+_i(z). 
\end{align*}
    
For $i\in \hat{I}$ and  $j\in I$
\begin{align*}
	  & (z-q^{\mu_\s(i)+\frac{1}{2}s_i+\frac{3}{2}B^\s_{i,j}}w)\phi^\pm_i(c^{-(1\pm 1)/2}z)x^+_j(w)=q^{B^\s_{i,j}}(z-q^{\mu_\s(i)+\frac{1}{2}s_i-\frac{1}{2}B^\s_{i,j}}w)x^+_j(w)\phi^\pm_i(c^{-(1\pm 1)/2}z),                                                                                      \\
	  & (z-q^{\mu_\s(i)+\frac{1}{2}s_i-\frac{1}{2}B^\s_{i,j}}w)\phi^\pm_i(c^{-(1\mp 1)/2}z)x^-_j(w)=q^{-B^\s_{i,j}}(z-q^{\mu_\s(i)+\frac{1}{2}s_i+\frac{3}{2}B^\s_{i,j}}w)x^-_j(w)\phi^\pm_i(c^{-(1\mp 1)/2}z).                          \end{align*}

	It is known that the subalgebra generated by the coefficients of $k^{\pm}_i(z)$, $x^\pm_i(z)$, $i\in I$, is isomorphic to $U_q\slhs$.
	
	The Heisenberg current commuting with $U_q\slhs$ is given by
	$$
	H^\pm (z)= H^{\pm1}\exp \big(\pm (q-q^{-1})\sum_{r>0}H_{\pm r}z^{\mp r}\big)=  \prod_{i=1}^N\left(\phi_{i}^\pm(zq^{-2\mu_{\bs s}(i)+s_i} )\right)^{s_i}.
	$$
	Note that $H=\prod_i {\phi_i}^{s_i}$ is central.
	  \medskip
	  
	  It is useful to write the relations between the Cartan currents $k_i^\pm(z)$ and $\phi_j^{\pm}(w)$:
	  \begin{align*} 
	  & \frac{c^{-1}q^{-\mu_\s(i)-\frac{1}{2}s_i+\frac{1}{2}B_{i,j}}z-w}{cq^{-\mu_\s(i)-\frac{1}{2}s_i+\frac{1}{2}B^\s_{i,j}}z-w}\phi^+_i(z)k^-_j(w)=\frac{c^{-1}q^{-\mu_\s(i)-\frac{1}{2}s_i-\frac{3}{2}B^\s_{i,j}}z-w}{cq^{-\mu_\s(i)-\frac{1}{2}s_i-\frac{3}{2}B^\s_{i,j}}z-w}k^-_j(w)\phi^+_i(z), \\
	  & \frac{c^{-1}q^{-\mu_\s(i)-\frac{1}{2}s_i-\frac{3}{2}B^\s_{i,j}}z-w}{cq^{-\mu_\s(i)-\frac{1}{2}s_i-\frac{3}{2}B^\s_{i,j}}z-w}\phi^-_i(z)k^+_j(w)=\frac{c^{-1}q^{-\mu_\s(i)-\frac{1}{2}s_i+\frac{1}{2}B^\s_{i,j}}z-w}{cq^{-\mu_\s(i)-\frac{1}{2}s_i+\frac{1}{2}B^\s_{i,j}}z-w}k^+_j(w)\phi^-_i(z). 
\end{align*}

\medskip
There is the \textit{homogeneous $\Z$-grading} of $\U_q\glhs$ given by
\begin{align*}
	\deg_{\delta}( x^\pm_{j,r})=r,\quad \deg_{\delta}(\phi_{j,r})=r, \quad \deg_{\delta}(\phi_j^{\pm 1})=\deg_{\delta}(c)=0 \quad  (j\in I,  r\in \Z'). 
\end{align*}
	Moreover, 
$\deg^v (X)=\left(\deg_1^v(X),\dots,\deg_{N-1}^v(X);\deg_{\delta}(X) \right)$,
	where $X\in \U_q\glhs$ and for $i\in I$,  
\begin{align*}
	  & \deg_i^v(x^\pm_{j,r})=\pm \delta_{i,j}, & \deg_i^v(\phi_{j,r})=\deg_i^v  (\phi_j^{\pm 1})=\deg_i^v(c)=0 &&(j\in I,  r\in \Z'),
\end{align*}
defines a $\Z^{m+n}$-grading of superalgebra $\U_q\glhs$.
	
We define $\widetilde{U}_q\widehat{\mathfrak{gl}}_{\s}$ to be the completion of $\U_q\glhs$ with respect to the homogeneous grading in the positive direction. Thus the elements of $\widetilde{U}_q\widehat{\mathfrak{gl}}_{\s}$ are series of the form $\sum_{j=-M}^\infty g_j$, with $g_j\in \U_q\glhs,$ $\deg_{\delta} g_j=j$. The algebra $\widetilde{U}_q\widehat{\mathfrak{sl}}_{\s}$ is defined in the same way.
	
\begin{lem}{\label{lem1}}
	We have an embedding
	\begin{align*}
		U_q\widehat{\mathfrak{gl}}_{\s} \rightarrow \widetilde{U}_q\widehat{\mathfrak{gl}}_{\s}. 
	\end{align*}
	\qed
\end{lem}
	
A $U_q\widehat{\mathfrak{gl}}_{\s}$-module $V$ is {\it admissible} if for any $v\in V$ there exist $M=M_v>0$ such that $xv=0$ for all $x \in U_q\widehat{\mathfrak{gl}}_{\s}$ with $\deg_\delta x> M $. Any admissible $U_q\widehat{\mathfrak{gl}}_{\s}$-module is also an $\widetilde{U}_q\widehat{\mathfrak{gl}}_{\s}$-module. 
	 
A vector $v$ in a $U_q\widehat{\mathfrak{gl}}_{\s}$-module is called {\it highest weight vector} if it is highest weight vector with respect to $U_q\widehat{\mathfrak{sl}}_{\s}\subset U_q\widehat{\mathfrak{gl}}_{\s}$
 and is annihilated by the positive modes of the current $H(z)$:
$$
e_iv=0, \qquad  \phi_iv=q^{\lambda_i}v \quad (i\in \hat I), \qquad  H_rv=0, \quad (r>0) .
$$
A $U_q\widehat{\mathfrak{gl}}_{\s}$-module $V$ is called {\it highest weight module} if $V$ is generated by the highest weight vector $v$, $V=U_q\widehat{\mathfrak{gl}}_{\s}v$.
	 
Highest weight $U_q\widehat{\mathfrak{gl}}_{\s}$-modules are admissible.

\medskip
	 
As before, define $U_q\widehat{\mathfrak{gl}}_{\bullet}=\bigoplus_{\s\in \Smn} U_q\widehat{\mathfrak{gl}}_{\s}$, $\widetilde{U}_q\widehat{\mathfrak{gl}}_{\bullet}=\bigoplus_{\s\in \Smn} \widetilde{U}_q\widehat{\mathfrak{gl}}_{\s}$, and
$\widetilde{U}_q\widehat{\mathfrak{sl}}_{\bullet}=\bigoplus_{\s\in \Smn} \widetilde{U}_q\widehat{\mathfrak{sl}}_{\s}$.

Let $\B_N^{fin}$ be the subgroup of $\B_N$ generated by $T_i$, $i\in I$, see Section \ref{ssbraid}. We call $\B_N^{fin}$ the {\it braid group associated with $\sl_{N}$}.

The ``fused currents'' technique can be used to define an action of the group $\B_N^{fin}$ on $\widetilde{U}_q\widehat{\mathfrak{sl}}_{\bullet}$ as follows. See \cite{FJMM} for a more detailed discussion of fused currents.

\begin{prop}
	For $i\in I$, $\s\in \Smn$, there exists an isomorphism $T_{i,\s}:\widetilde{U}_q\widehat{\mathfrak{sl}}_{\s} \rightarrow \widetilde{U}_q\widehat{\mathfrak{sl}}_{\sigma_i \s}$ given by
	\begin{align}
		&T_{i,\s}(x^+_{i+1}(z))=s_{i+1}\lim_{z'\to z} \left(1-\dfrac{z}{z'}\right)x^+_{i}(q^{-s_i}z')x^+_{i+1}(z),\\
		&T_{i,\s}(x^+_{i-1}(z))=s_{i+1}\lim_{z'\to z} \left(1-\dfrac{z}{z'}\right)x^+_{i}(q^{-s_{i+1}}z')x^+_{i-1}(z),\\
		&T_{i,\s}(x^-_{i+1}(z))=s_i\lim_{z'\to z} \left(1-\dfrac{z}{z'}\right)x^-_{i+1}(z')x^-_{i}(q^{-s_i}z),\\
		&T_{i,\s}(x^-_{i-1}(z))=s_i\lim_{z'\to z} \left(1-\dfrac{z}{z'}\right)x^-_{i-1}(z')x^-_{i}(q^{-s_{i+1}}z),\\
		&T_{i,\s}(k^\pm_{i+1}(z))=k^\pm_{i}(q^{-s_{i}}z)k^\pm_{i+1}(z),\\
		&T_{i,\s}(k^\pm_{i-1}(z))=k^\pm_{i}(q^{-s_{i+1}}z)k^\pm_{i-1}(z),\\
		&T_{i,\s}(x^+_{i}(z))=s_ix^-_{i}(c^{-1}q^{-s_i-s_{i+1}}z)k^+_i(c^{-1}q^{-s_i-s_{i+1}}z)^{-1},\\
		&T_{i,\s}(x^-_{i}(z))=s_{i+1}k^-_i(c^{-1}q^{-s_i-s_{i+1}}z)^{-1}x^+_{i}(c^{-1}q^{-s_i-s_{i+1}}z),\\
		&T_{i,\s}((k^\pm_{i}(z))=k^\pm_{i}(q^{-s_i-s_{i+1}}z)^{-1},\\
		  & T_{i,\s}(A_j(z))=A_j(z) & (j\neq i, i\pm 1\, A=x^\pm, k^\pm). 
	\end{align}
	
	Moreover, the isomorphisms $T_{i,\s}$ satisfy
	\begin{align*}
		&(\eta\, T_{i}\, \eta)_\s=(T_{i,\sigma_{i}\s})^{-1},\\
		  & (T_iT_{j})_\s=(T_jT_{i})_\s       & (j\neq i,i\pm 1), \\
		  & (T_jT_iT_{j})_\s=(T_iT_jT_{i})_\s & (j=i\pm 1).       
	\end{align*}
	\begin{proof}
		The proposition is checked directly. In some cases it is useful to consider the following alternative formula for the fused currents. This formula for $n=0$, i.e., in the purely even case, was obtained in \cite{DK}, see also \cite{M2}.
		 
		Set $\Res_{z}(\sum_{i\in\Z} a_iz^{-i})=a_1$. Then,
		
		\begin{align*}
			  & \lim_{z'\to z} \left(1-\dfrac{z}{z'}\right)x^+_{i}(q^{A_{i,j}}z')x^+_{j}(z)=\Res_{w}\left( x^+_{i}(q^{A_{i,j}}z)x^+_{j}(w)-(-1)^{|i||j|}\dfrac{q^{2A_{i,j}}z-w}{q^{A_{i,j}}(z-w)} x^+_{j}(w)x^+_{i}(q^{A_{i,j}}z) \right),   \\
			  & \lim_{z'\to z} \left(1-\dfrac{z}{z'}\right)x^-_{i}(z')x^-_{j}(q^{A_{i,j}}z)=\Res_{w}\left( x^-_{i}(z)x^-_{j}(q^{A_{i,j}}w)-(-1)^{|i||j|}\dfrac{q^{-2A_{i,j}}z-w}{q^{-A_{i,j}}(z-w)} x^-_{j}(q^{A_{i,j}}w)x^-_{i}(z) \right), 
		\end{align*}
		where the rational functions must be expanded in the region $|w|>|z|$.
	\end{proof}
\end{prop}

Define
\begin{align*}
	  & T_i=\bigoplus_{\s\in \Smn} T_{i,\s} &   & (i\in \hat{I}). 
\end{align*}

\begin{cor}
	The automorphisms $T_i$, $i\in I$, define an action of $\B_N^{fin}$ on $\widetilde{U}_q\widehat{\mathfrak{sl}}_{\bullet}$.
\end{cor}

We are now able to construct an evaluation homomorphism from $\Es$ to $\widetilde{U}_q\widehat{\mathfrak{gl}}_{\s}$. The evaluation map for the standard parity sequence was constructed in \cite{BM} without explicitly using the above braid action. The evaluation map for all choices of $\s$ given in the following theorem coincides with the previous one if $\s$ is the standard parity sequence. This construction is similar to the $n=0$ case, which was done in \cite{M2}.

\begin{thm}
	Fix $u\in \mathbb{C}^\times$. The following map is a surjective homomorphism of superalgebras $\eva_u^\s : \Es \rightarrow \widetilde{U}_q\widehat{\mathfrak{gl}}_{\s}$ with  $ C^2=q_3^{m-n}$:
	\begin{align*}
		&K\mapsto 1,\quad C\mapsto c,\\
		  & E_i(z)\mapsto x^+_i(d^{-\mu_\s(i)}z),\quad F_i(z)\mapsto x^-_i(d^{-\mu_\s(i)}z),\quad K_i^\pm(z)\mapsto k_i^\pm(d^{-\mu_\s(i)}z) \quad & (i\in I), \\
		& E_0(z)\mapsto u^{-1}\,\phi^-_0(c^{-2}q^{2(n-m)}z)\times(T_{N-1}\dots T_1)_{\tau \s}(x^+_1(zq^{s_N}))\times \phi^+_0(c^{-1}q^{2(n-m)}z)^{-1},\\
		&F_0(z)\mapsto u\,\phi^-_0(c^{-1}q^{2(n-m)}z)^{-1}\times (T_{N-1}\dots T_1)_{\tau \s}(x^-_1(zq^{s_N}))\times \phi^+_0(c^{-2}q^{2(n-m)}z),\\
		&K^\pm_0(z)\mapsto (T_{N-1}\dots T_1)_{\tau \s}(k^\pm_1(zq^{s_N}))\times \phi^\pm_0(c^{-2}q^{2(n-m)}z)\phi^\pm_0(q^{2(n-m)}z)^{-1}.
	\end{align*}
	\begin{proof}
		The proof of the theorem is the same as in \cite{BM} with the appropriate changes on the shifts in the formulae.
	\end{proof}
\end{thm}

Note that if $X\in \Es$ has grading  $\deg X=(d_0,d_1,\dots,d_{N-1},d_\delta)$ then the grading of the image is given by $\deg^v(ev_u^\s(X))=(d_1-d_0,d_2-d_0,\dots,d_{N-1}-d_0;d_\delta)$.

\end{document}